\documentclass[a4paper,11pt]{amsart}

\usepackage{amssymb,amsmath,color}
\usepackage{graphics}

\newtheorem{theorem}{Theorem}[section]
\newtheorem{corollary}[theorem]{Corollary}

\newtheorem{proposition}[theorem]{Proposition}
\newtheorem{conjecture}[theorem]{Conjecture}

\newtheorem{lemma}[theorem]{Lemma}
\newtheorem{problem}[theorem]{Problem}

\newtheorem*{lemma*}{Lemma}

\theoremstyle{definition}

\newcommand{\N}{\mathbb{N}}

\newcommand{\R}{\mathbb{R}}

\newcommand{\F}{\mathbb{F}}
\renewcommand{\t}{\mathfrak{t}}

\newcommand{\E}{\mathbb{E}}

\renewcommand{\epsilon}{\varepsilon}
\renewcommand{\emptyset}{\varnothing}

\newcommand{\mbf}[1]{\mathbf{#1}}

\DeclareMathOperator{\Stab}{Stab}

\DeclareMathOperator{\im}{im}
\DeclareMathOperator{\Tr}{Tr}
\DeclareMathOperator{\sgn}{sgn}

\newcommand*{\symdiff}{\,\scalebox{0.75}{$\bigtriangleup$}\,}

\newcommand{\Ind}{\big\uparrow}
\newcommand{\Res}{\big\downarrow}
\newcommand{\ind}{\!\!\uparrow}
\newcommand{\res}{\!\!\downarrow}

\renewcommand{\u}[1]{{}^{#1}\hspace*{-0.5pt}}

\renewcommand{\S}{\mathfrak{S}}
\newcommand{\A}{\mathfrak{A}}

\newcommand{\EE}{N} 
\renewcommand{\E}{C} 
\newcommand{\D}{F} 

\begin{document}

\title[The vertex of $S^{(kp-p,1^p)}$]{Sylow subgroups of symmetric and alternating groups and the
vertex \\ of $S^{(kp-p,1^p)}$ in characteristic $p$}

\author{Eugenio Giannelli}
\address[E. Giannelli]{Department of Mathematics, Royal Holloway, University of London, United Kingdom.}
\email{Eugenio.Giannelli.2011@live.rhul.ac.uk}

\author{Kay Jin Lim}
\address[K. J. Lim]{Division of Mathematical Sciences, Nanyang Technological University, SPMS-MAS-03-01, 21 Nanyang Link, Singapore 637371.}
\email{limkj@ntu.edu.sg}

\author{Mark Wildon}
\address[M. Wildon]{Department of Mathematics, Royal Holloway, University of London, United Kingdom.}
\email{mark.wildon@rhul.ac.uk}

\thanks{Part of this work was done while the third author was visiting the Institute for Mathematical Sciences, National University of Singapore in 2013. The visit was supported by the Institute.}
\thanks{The second author is supported by Singapore Ministry of Education AcRF Tier 1 grant RG13/14.}

\subjclass[2010]{Primary 20C30, Secondary 20B30, 20C20, 20D06, 20D20}
\keywords{Specht module, vertex, symmetric group, alternating group, Sylow subgroup}

\date{\today}

\begin{abstract} We show that the Sylow $p$-subgroups of a symmetric group, respectively an alternating group, are characterized as the $p$-subgroups containing all elementary abelian $p$-subgroups up to conjugacy of the symmetric group, respectively the alternating group. We apply the characterization result for symmetric groups to compute the vertices of the hook Specht modules associated to the partition $(kp-p,1^p)$ under the assumption that $k \equiv 1$ mod $p$ and $k \not\equiv 1$ mod $p^2$.
\end{abstract}

\maketitle
\thispagestyle{empty}

\section{Introduction}

In this paper we prove two main theorems. The first
gives a new characterization of the Sylow $p$-subgroups of the symmetric groups
$\S_n$ and alternating groups $\A_n$.

\begin{theorem}\label{thm:char}
Suppose that $G$ is either $\A_n$ or $\S_n$. Let $p$ be a prime and let
$Q \le G$ be a $p$-subgroup. Then $Q$ is a Sylow $p$-subgroup
of $G$ if and only if $Q$ contains a $G$-conjugate of every elementary abelian
$p$-subgroup of $G$.
\end{theorem}

It is easily seen that this characterization of Sylow subgroups
does not hold in a general finite group.
When $n$ is
large compared to $p$, the Sylow $p$-subgroups of $\S_n$ and
$\A_n$ are much  larger than the elementary abelian $p$-groups they contain.
For example, if $n$ is a power of $2$ then a
Sylow $2$-subgroup of $\S_{n}$ has order~$2^{n-1}$, whereas the order of the
largest elementary abelian $2$-subgroup
of~$\S_{n}$ is $2^{n/2}$.
This gives another indication that Theorem~\ref{thm:char} is a non-trivial result.
We state some open problems suggested by this theorem in the final section of the paper.

Theorem~\ref{thm:char} is motivated by an application to vertices of indecomposable modules
for the symmetric group.  We recall the key definitions.
Let $F$ be a field of prime characteristic $p$, let $G$
be a finite group and let $M$ be an indecomposable $FG$-module. A subgroup
$Q$ of $G$ is said to be a \emph{vertex} of $M$ if there exists a $FQ$-module
$U$ such that $V$ is a direct summand of $U\ind^G$
and~$Q$ is minimal with this property.
This definition was introduced by Green in~\cite{Green59},
where he showed that the vertices of $M$ form a single conjugacy class of $p$-subgroups of $G$. We note that the vertices of $M$ do not change under field extensions of $F$ \cite[Theorem 1.21]{HuppertBlackburn}. Vertices play a key role in the Green correspondence \cite{Green63}, and
in Alperin's Weight Conjecture \cite{AlperinWC}. But despite the importance
of vertices to local representation theory, there are relatively few cases where
the vertices of particular families of modules have been computed.
Our second main theorem is as follows.


%

\begin{theorem}\label{thm:vertex}
Let $p$ be an odd prime. Suppose that $k \equiv 1$ mod $p$ and that
$k \not\equiv 1$ mod $p^2$.
The Specht module $S^{(kp-p,1^p)}$, defined over a field of
prime characteristic $p$, has a Sylow $p$-subgroup
of~$\S_{kp}$ as a vertex.
\end{theorem}

We prove Theorem~\ref{thm:vertex}
as a corollary of the following result, which is
of independent interest.

\begin{theorem}\label{thm:elem}
Let $p$ be an odd prime
and let $E$ be an elementary abelian $p$-subgroup of $\S_{kp}$. If either
\begin{enumerate}
\item[(i)] $E$ has at least two orbits of size $p$ on $\{1,\ldots, kp\}$, or
\item[(ii)] $E$ has at least one orbit of size $p^2$ on $\{1,\ldots,kp\}$,
\end{enumerate}
then $E$ is contained in a vertex of the Specht module $S^{(kp-p,1^p)}$,
defined over an algebraically closed field of characteristic $p$.
\end{theorem}

The proof of Theorem~\ref{thm:elem} uses
two key techniques:
generic Jordan type for
modules over elementary abelian $p$-groups (see \cite{Friedlander-Pevtsova-Suslin,Wheeler}),
and the Brauer correspondence for modules,
as developed by Brou{\'e} in \cite{BrouePPerm}.
The isomorphism $S^{(kp-p,1^p)} \cong \bigwedge^p S^{(kp-1,1)}$ (see Section~\ref{sec:hook})
is used throughout the calculations.


\subsection*{Background}
We now survey the existing results on vertices of Specht modules, emphasising the different techniques
that have been used.
We begin with results that, like Theorem~\ref{thm:vertex}, apply to Specht modules
labelled by partitions of a special form. Note that,
by \cite[Corollary 13.18]{James}, Specht modules are always indecomposable
in odd characteristic.

The simplest case is a Specht module $S^\lambda$ in a block of abelian defect. 
In this case $S^\lambda$ has dimension coprime to $p$ and by \cite[Corollary 1]{Green59}, its vertices
are the defect groups of its block; these are elementary abelian.

In \cite[Theorem 2]{WildonCycSpechtVertices} the
third author used the Brauer correspondence to show that when $p$ does
not divide $n$ and $S^{(n-r,1^r)}$ is indecomposable, its vertex is a Sylow
$p$-subgroup of $\S_{n-r-1}\times \S_r$.
When $p=2$ and $n$ is even it was shown by
Murphy and Peel in \cite[Theorem~4.5]{MurphyPeel} that
$S^{(n-r,1^r)}$ is indecomposable and has
a Sylow $2$-subgroup of $\S_n$ as a vertex. A key step in their proof is
their Theorem 4.3, which states that
if $Q$ is an elementary abelian $2$-subgroup of $\S_n$ of rank $n/2$ then $S^{(n-r,1^r)}\res_{\hspace*{1pt}Q}$
is indecomposable. 
The analogue of this result for odd primes is in general false:
for example, when $p = 5$,
calculations with the computer algebra package {\sc Magma} \cite{Magma}
show that if $Q = \langle (1,2,3,4,5), (6,7,8,9,10) \rangle$ then
$S^{(8,1,1)}\res_{\hspace*{1pt}Q}$ is decomposable. It therefore seems
impossible to extend the methods of \cite{MurphyPeel} to the case when $p$ is odd and $p$
divides $n$.

The dimension of $S^{(kp-r,1^r)}$ is $\binom{kp-1}{r}$, which is
coprime to $p$ whenever $r < p$. In these cases $S^{(kp-r,1^r)}$ has a full Sylow
$p$-subgroup of $\S_{kp}$ as a vertex again by \cite[Corollary 1]{Green59}. When $r=p$
the dimension is divisible by $p$ if and only if $k \equiv 1$ mod $p$,
and divisible by $p^2$ if and only if $k \equiv 1$ mod $p^2$.
Theorem~\ref{thm:vertex} therefore deals with the smallest case that
cannot be decided solely by dimension arguments.

When $p=2$ the Specht module $S^{(n-2,2)}$ is always indecomposable. When~$n$ is even
or  $n \equiv 1$ mod $4$, it was shown in \cite[Theorem 1]{DanzErdmann} that $S^{(n-2,2)}$
has a Sylow $2$-subgroup of $\S_n$ as a vertex, except when $n=4$, when the unique vertex
of $S^{(2,2)}$ is the normal Sylow $2$-subgroup of $\A_4$. When $n \equiv 3$ mod $4$,
any Sylow $2$-subgroup of $\S_{n-5} \times \S_2 \times \S_2$ is a vertex.
When $p \ge 3$ the third author showed in \cite[Theorem 8.1]{WildonBound}
that $S^{(n-2,2)}$ has a Sylow $p$-subgroup
of the defect group of its block as a vertex.

In the main result of \cite{WildonBound} the third author used the Brauer correspondence to determine a subgroup
contained in the vertex of each indecomposable Specht module. This result was later
improved by the first author in \cite{GiannelliBound}.
In \cite{LimAbelian} the second author used  the complexity
theory of modules to prove a number of classification theorems
on Specht modules with abelian vertices. In particular it was shown that if $p \ge 3$ and $S^\lambda$
has an abelian vertex of $p$-rank $m$ then $\lambda$ is a $p^2$-core of $p$-weight $m$, and
that $S^{(p^p)}$ has a Sylow $p$-subgroup of $\S_{p^2}$ as a vertex.

In \cite{HemmerIrreducible} Hemmer proved that, when $p \ge 3$, a Specht module is simple
if and only if it is a signed Young module, as defined by Donkin \cite[page 651]{Donkin}.
By \cite[(3) page 675]{Donkin}, it follows that the vertices of simple Specht modules in odd characteristic
are 
Sylow $p$-subgroups of Young subgroups of symmetric groups.
For example, when $n$ is not divisible by $p$, the Specht module $S^{(n-r,1^r)}$ is isomorphic to the
signed Young module \hbox{$Y( (n-r,1^{r-ps})| p(s) )$} where $\lfloor r/p \rfloor = s$,
and so its vertex may also be found using Donkin's result. In general it is not
possible to apply Donkin's result in this way
because Hemmer's proof does not explicitly determine a pair of partitions labelling
a signed Young module isomorphic to a given simple Specht module.

Finally we note that there has also been considerable work on the vertices of simple
modules for symmetric groups.
We refer the reader to \cite{DanzKulshammerSurvey} for a survey.

\subsection*{Outline}
In Section~\ref{sec:charS} we prove Theorem~\ref{thm:char} for symmetric groups. Then
in Section~\ref{sec:charA} we use this result and further arguments to prove Theorem~\ref{thm:char} for alternating groups.
In Section~\ref{sec:elemtovertex} we give a short proof, using Theorem~\ref{thm:char},
that Theorem~\ref{thm:elem} implies Theorem~\ref{thm:vertex}.

We then turn to the proof of Theorem~\ref{thm:elem}.
In Section~\ref{sec:hook} we collect the necessary background results on hook-Specht modules.
We review the Brauer correspondence in Section~\ref{sec:Brauer} and give a useful
general result on the Brauer homomorphism for monomial modules. The notion of generic
Jordan type is presented in Section~\ref{sec:gJt}. Then in Sections~\ref{sec:proofelemi}
and~\ref{sec:proofelemii} we prove Theorem~\ref{thm:elem}, under hypotheses~(i) and~(ii),
respectively.

We end in Section~\ref{sec:problems} by stating some open problems related to our three main theorems.

\section{Proof of Theorem~\ref{thm:char} for symmetric groups}
\label{sec:charS}

For background on Sylow subgroups
of symmetric groups and their construction as product of iterated wreath products, we
refer the reader to \cite[(4.1.20)]{JK}.

We need the following notation. Let $\S_A$ denote the symmetric group on a finite set $A$.
For each $A \subset \N$ such that $A$ has size $p^d$ where $d \in \N$,
let $E_A$ denote a fixed elementary abelian subgroup of $\S_A$ of order $p^d$. All such subgroups
are permutation isomorphic to the regular representation of an elementary
abelian group of order $p^d$, and so are conjugate in $\S_A$. We highlight that the precise choice of $E_A$ is irrelevant.


The main work in the proof of Theorem~\ref{thm:char} comes in the case when $n=mp$ for some $m \in \N$. Suppose that
\begin{equation}\label{eq:dec}
n = m_1 p + m_2 p^2 + \cdots + m_r p^r
\end{equation}
where $m_i \in \N_0$ for each $i$ and $m_r \not= 0$. Note that this is not necessarily
the $p$-adic expansion of $n$. For $i \in \{1,\ldots,r\}$ and each $j$
such that $1 \le j \le m_i$, let 
\[ z^i_j = (j-1)p^i  + \sum_{s=1}^{i-1} m_sp^s \]
and let 
$B^i_j = \{ z^i_j + 1, \ldots, z^i_j + p^i\}$.
Note that $|B^i_j| = p^i$ for all $j$ such that $1 \le j \le m_i$. 
If $m_i = 0$ then $B^i_j = \varnothing$ and $E_{B^i_j}$ is the trivial group.
For each expression for $n$ as in~\eqref{eq:dec}, define
an elementary abelian $p$-subgroup of $\S_n$ by
\[ E(m_1,\ldots, m_r) = \prod_{i=1}^r \prod_{j=1}^{m_i} E_{B^i_j}. \]

\begin{proposition}[{\cite[1.3 Chapter VI]{Adem-Milgram}}]\label{prop:max}
The $p$-subgroups $E(m_1,\ldots,m_r)$ where $m_1, \ldots, m_r$ satisfy \eqref{eq:dec}
form a complete irredundant set of maximal elementary abelian $p$-subgroups of $\S_{mp}$ up to $\S_{mp}$-conjugacy.
\end{proposition}

In fact each $E(m_1,\ldots,m_r)$ is maximal even as an abelian subgroup of~$\S_{mp}$; we leave this remark
as an exercise to the reader.


The subgroup $E(m)$ generated by $m$ disjoint $p$-cycles is critical to our argument.
To emphasise the role played by this subgroup, we denote it by $\E(m)$ below.
Informally stated, the next lemma says that $\E(m)$ is the normal base group in any
Sylow $p$-subgroup that contains it.

\begin{lemma}\label{lemma:Enormal}
If $Q$ is a $p$-subgroup of $\S_{mp}$ containing $\E(m)$ then $\E(m) \unlhd Q$. Moreover $Q/\E(m)$
acts faithfully on the orbits of $\E(m)$.
\end{lemma}
\begin{proof}
Let $P$ be a Sylow $p$-subgroup of $\S_{mp}$ containing $Q$. By the construction
of Sylow $p$-subgroups in \cite[(4.1.20)]{JK}, $P$ has a normal base group $B$
generated by $m$ disjoint $p$-cycles $\tau_1, \ldots, \tau_m$. If $\sigma \in P$
then $\sigma$ permutes the orbits of the $\tau_k$ as blocks for its action.
It follows that if $\sigma \in \E(m)$ is a $p$-cycle then $\sigma \in \langle \tau_k \rangle$ for some $k$.
Hence $\E(m) = B$ and $\E(m) \unlhd P$. Moreover, if $\sigma \in P$ fixes each
orbit of $\E(m)$ setwise then $\sigma$ is a product of disjoint $p$-cycles each contained
in $\E(m)$. Therefore $P/\E(m)$ acts faithfully on the orbits of~$\E(m)$.
\end{proof}

In particular, Lemma~\ref{lemma:Enormal} determines how an elementary abelian $p$-subgroup that normalizes $\E(m)$
can act on the orbits of $\E(m)$.


\begin{lemma}\label{lemma:regular}
Let $d \in \N$ and let $\sigma\in\S_{p^d}$. Let $\E(p^{d-1})$ be as defined when $n = p^d$ and
let $X$ be the set of orbits of $\E(p^{d-1})$.
Let $A=\{1,2,\ldots,p^d\}$.
Suppose that a $p$-subgroup $P$ of $\S_{p^d}$ contains both $\u\sigma E_A$ and $\E(p^{d-1})$. Then
$\u\sigma E_A\E(p^{d-1})/\E(p^{d-1})$ acts regularly on $X$. 
\end{lemma}

\begin{proof}
By Lemma~\ref{lemma:Enormal} in the case $n = p^d$, the subgroup $\E(p^{d-1})$ is normal in $P$ and
$\u\sigma E_A\E(p^{d-1})/\E(p^{d-1})$ acts faithfully on $X$. 
Since $E_A$ acts transitively on $\{1,\ldots,p^d\}$,
the action of $\u\sigma E_A\E(p^{d-1})/\E(p^{d-1})$ on $X$ is transitive. Since
$E_A$ is abelian, it follows that $\u\sigma E_A\E(p^{d-1})/\E(p^{d-1}) \cong
\u\sigma E_A / (\E(p^{d-1}) \cap \u\sigma E_A)$ is also abelian and hence it acts regularly on $X$.
\end{proof}

We are now ready to prove Theorem~\ref{thm:char}.

\begin{proof}[Proof of Theorem \ref{thm:char} when $G=\S_{n}$]
We prove the theorem by induction on~$n$.
Let $Q \le \S_n$ be a $p$-subgroup containing an $\S_n$-conjugate
of every elementary abelian $p$-subgroup of $\S_n$.
Let $n = mp +c$ where $m \in \N$ and $0 \le c < p$.
If $c \ge 1$ then by conjugating $Q$
by an element of $\S_n$ we may assume that $Q \le \S_{mp}$. By induction
we get that $Q$ is a Sylow $p$-subgroup of $\S_{mp}$, and
hence a Sylow $p$-subgroup of $\S_n$.
We may therefore assume that $n=mp$, and, by induction,
that the theorem holds true for $\S_m$.

By replacing $Q$ with a $\S_{mp}$-conjugate, we may assume that $\E(m)$ is contained in $Q$. Let $X = \{B_1^1, \ldots, B_m^1\}$ be the set of orbits of $\E(m)$. By Lemma~\ref{lemma:Enormal} we have that $\E(m) \unlhd Q$ and $Q/\E(m)$ acts faithfully on $X$.


Let $\overline{H} \le \S_X$ denote
the image of a subgroup $H$ of $Q$ under the quotient map $Q \to Q/\E(m)$,
regarding $Q/\E(m)$ as a subgroup of $\S_X$.
We aim to prove that $\overline{Q}$ contains a $\S_X$-conjugate of every elementary abelian $p$-subgroup of $\S_X$.

Let $m = \ell_0 + \sum_{i=1}^t \ell_i p^i$ such that $\ell_t\neq 0$.
By assumption there exists $\sigma \in \S_{mp}$
such that $Q$ contains $\u\sigma E(\ell_0,\ldots, \ell_t)$. (Note that because of the index shift,
this subgroup has exactly $\ell_{i-1}$ orbits of size $p^i$ on $\{1,\ldots, mp\}$.)
We have
\[ \u\sigma E(\ell_0,\ldots, \ell_t) = \prod_{i=1}^{t+1} \prod_{j=1}^{\ell_{i-1}} E_{\sigma B^i_j}. \]
Fix $i \in \{1, \ldots, t+1\}$ and $j \in \{1,\ldots, \ell_{i-1}\}$.
Let $\Omega^i_j$ be the set of all $k \in \{1,\ldots, m\}$ such
that the orbit $B^1_k$ of $\E(m)$ is contained in the 
support of~$E_{\sigma B^{i}_j}$. 
Let
$X^i_j = \{B^1_k : k \in \Omega^i_j \}$. Thus $X$ is the disjoint union of the $X^{i}_j$, the support
of $E_{\sigma B^{i}_j}$ is $\bigcup_{k \in \Omega^i_j} B^1_k$
and $\overline{E_{\sigma B^i_j}}$ acts on $X^i_j$.
By Lemma~\ref{lemma:regular} this action is regular.
It follows that the elementary abelian $p$-subgroup
\[ \overline{\u\sigma E(\ell_0,\ldots, \ell_t)} \le \S_X \]
has exactly $\ell_{i}$ orbits on $X$ of size $p^{i}$ for each $i \in \{1,\ldots, t+1\}$. Since the $\ell_i$
were freely chosen,
we see that
$\overline{Q}$ contains a $\S_X$-conjugate of every elementary abelian $p$-subgroup of $\S_X$.
It now follows by induction that
 $\overline{Q}$ acts on $X$ as a Sylow $p$-subgroup of~$\S_X$.

If $p^a$ is the order of a Sylow $p$-subgroup of $\S_m$ then
\[ |Q| = |\E(m)|\, |\overline{Q}| = p^{m+a},\]
which is the order of a Sylow $p$-subgroup of $\S_{mp}$.
The theorem follows.
\end{proof}

\section{Proof of Theorem~\ref{thm:char} for alternating groups}
\label{sec:charA}

\subsection*{Case $p \ge 3$}
In this case any $p$-subgroup of $\S_n$ is a $p$-subgroup of $\A_n$.
Hence if $Q \le \A_n$ contains an $\A_n$-conjugate of every elementary abelian $p$-subgroup
of $\A_n$ then $Q$ contains an $\S_n$-conjugate of every elementary abelian $p$-subgroup
of $\S_n$, and so by Theorem~\ref{thm:char} for $\S_n$, we get that $Q$ is a Sylow $p$-subgroup
of $\S_n$.

\subsection*{Case $p=2$}
Let $n = 2m$.
We need the
classification of all maximal elementary abelian $2$-subgroups of $\A_{n}$ up to $\S_{n}$-conjugacy.
Suppose that
\begin{equation}\label{eq:dec2}
 n = 2m_1 + 2^2m_2 + \cdots + 2^r m_r
\end{equation}
where $m_i \in \N_0$ for each $i$, $m_r \not=0$ and $m_1 \not= 2$. For each such expression define
\[ F(m_1,m_2,\ldots, m_r) = E(m_1,m_2,\ldots, m_r) \cap \A_n. \]
The reason for excluding the case $m_1 = 2$ is that $E(2,m_2,\ldots,m_r)\cap \A_n$ is properly contained in $F(0,m_2+1,m_3,\ldots,m_r)$.
Note also that $F(m_1,m_2,\ldots,m_r)$ has exactly two fixed points on $\{1,\ldots,n\}$ if $m_1=1$.


\begin{proposition}[{\cite[Proposition 5.2]{AKMU}}]
\label{prop:maxalt} The subgroups $F(m_1,\ldots,m_r)$ where 
$m_1,\ldots, m_r$ satisfy~\eqref{eq:dec2} form a complete irredundant set
of maximal elementary 
abelian $2$-subgroups of $\A_{2m}$ up to $\S_{2m}$-conjugacy.
\end{proposition}

In \cite{AKMU} the authors remark that the proposition follows from the analogous result for $\S_n$.
We take this opportunity to provide a complete proof. The following lemma is required.

\begin{lemma}\label{lemma:FsubgroupE}
Suppose that $m_1, \ldots, m_r$ are as in~\eqref{eq:dec2} and $m_1', \ldots, m_s'$ are as in~\eqref{eq:dec}.
If $F(m_1,\ldots,m_r)\le \u\sigma E(m_1',\ldots,m_s')$ for some $\sigma\in\S_{2m}$
then $r=s$ and $m_i=m_i'$ for all $i \in \{1,\ldots, r\}$.
\end{lemma}

\begin{proof}
If $m_1 = 0$ then $F(0,m_2, \ldots,m_r) = E(0,m_2,\ldots,m_r)$ is maximal and by Proposition~\ref{prop:max}
we have $m_i = m_i'$ for all $i \in \{2,\ldots, r\}$. Hence $m_1' = 0$ and the lemma holds
in this case. If $m_1 = 1$ then $F(1,m_2,\ldots,m_r)$ is equal to the subgroup $E(0,m_2,\ldots,m_r)$. Since $E(0,m_2,\ldots,m_r)$ has exactly two fixed points and no orbit of size $2$ on $\{1,\ldots,2m\}$, we have $m_1'=1$. It follows that $\sigma$ fixes the set $\{1,2\}$ and hence $\u\sigma E(1,m_2,\ldots,m_r)\le E(m_1',\ldots,m_s')$. We may again apply Proposition \ref{prop:max}.

Suppose that $m_1 \ge 3$. Let $\tau_k = (2k-1,2k)(2k+1,2k+2)$
for $k \in \{1, \ldots, m_1 -1\}$. Let $B^2_j$ for $j \in \{1,\ldots,m_2'\}$
be as defined after~\eqref{eq:dec} for the subgroup $E(m_1',\ldots,m_s')$.
If there exist $j$ and $k$ such that $\tau_k \in \u\sigma E_{B^2_j}$ then
$\sigma(B^2_j) = \{2k-1,2k,2k+1,2k+2\}$ and an element of $\u\sigma E(m_1',\ldots,m_s')$
moves $2k-1$ if and only if it moves $2k+1$. But if $k > 1$ then $\tau_{k-1} \in F(m_1,\ldots,m_r)$
and otherwise $k=1$ and $\tau_2 \in F(m_1,\ldots,m_r)$, so this is a contradiction.
Hence all the transpositions $(2k-1,2k)$ for $1 \le k \le m_1$ lie in $E(m_1',\ldots, m_s')$
and so $m_1' \ge m_1$. On the other hand, $F(m_1,\ldots,m_r)$ has $m_1$ orbits
of size $2$ on $\{1,\ldots, 2m\}$ and $\u\sigma E(m_1',\ldots,m_r')$ has $m_1'$ orbits
of size $2$ on $\{1,\ldots, 2m\}$; since $F(m_1,\ldots,m_r) \le \u\sigma E(m_1',\ldots,m_r')$
we have $m_1 \ge m_1'$. Therefore $m_1 = m_1'$. Hence there exist $\tau, \tau' \in \S_{2m}$ such that
\begin{align*}
F(m_1,m_2,\ldots,m_r) &= F(m_1) \times \u\tau F(0,m_2,\ldots,m_r) \\
E(m_1',m_2',\ldots,m_s') &= \u\sigma E(m_1) \times \u{\tau'} E(0,m_2',\ldots,m_s')
\end{align*}
and $\u\tau F(0,m_2,\ldots,m_r) \le \u{\tau'} E(0,m_2',\ldots,m_s')$.
Both groups have support of size $2m - 2m_1$.
Hence, by the special case $m_1=0$ we get
$r=s$ and $m_i = m_i'$ for all $i \in \{2,\ldots, r\}$.
\end{proof}

\begin{proof}[Proof of Proposition~\ref{prop:maxalt}]
Let $F$ be an elementary abelian $p$-subgroup of $\A_{2m}$.
By Proposition~\ref{prop:max}
there exist $m_1', \ldots, m_s'$ and $\sigma \in \S_{2m}$ such that
\[F \le \u\sigma E(m_1',\ldots,m_s')\cap \A_{2m}=\u\sigma(E(m_1',\ldots,m_s')\cap \A_{2m}).\] If $m_1' \not = 2$ then $E(m_1',\ldots,m_s') \hskip0.5pt\cap\hskip0.5pt \A_{2m} = F(m_1',\ldots,m_s')$. In the remaining case we have $E(2,m_2',\ldots,m_s') \hskip0.5pt\cap\hskip0.5pt \A_{2m} \le
F(0,m_2'+1,\ldots,m_s')$.
Suppose further that $F=F(m_1,\ldots,m_r)$. By Lemma \ref{lemma:FsubgroupE}, we have $r=s$ and $m_i=m_i'$ for all $i\in \{1,\ldots,r\}$. Therefore the subgroups $F(m_1,\ldots,m_r)$
for $m_1,\ldots,m_r$ satisfying~\eqref{eq:dec2} form a complete set
of maximal elementary abelian 
$2$-subgroups of $A_{2m}$ up to $\S_{2m}$-conjugacy.
It is clear from the orbits of these subgroups on $\{1,\ldots,2m\}$
that no two of them are conjugate.
\end{proof}

We are now ready to prove Theorem~\ref{thm:char} in the alternating group case.

\begin{proof}[Proof of Theorem~\ref{thm:char} when $G=\A_n$ and $p=2$]

As in the proof of Theorem \ref{thm:char} for $\S_n$ we reduce to the case when $n = 2m$.
The cases $m=1$ and $m=2$ are easily checked, so we may assume that $m \ge 3$.

Let $Q \le \A_{2m}$ be a $2$-subgroup containing an $\A_{2m}$-conjugate of every
elementary abelian $2$-subgroup of $\A_{2m}$. Without loss of generality
we may assume that $\D(m)$ is contained in $Q$. Let $P$ be a Sylow $2$-subgroup
of $\S_{2m}$ containing $Q$. As in the proof of Lemma~\ref{lemma:FsubgroupE},
let $\tau_k = (2k-1,2k)(2k+1,2k+2)\in Q$ for $1 \le k < m$. Let $(x,y)$ be a
transposition in $P$ with $x < y$. If $y\not= x +1$ then there exists
$k$ such that exactly one of $x,y$ lies in $\{2k-1,2k,2k+1,2k+2\}$.
It is easily checked that in this case $[(x,y), \tau_k]$ is a $3$-cycle. 
So $y=x+1$. Moreover, if $x = 2k$ and $y=2k+1$ for some~$k$ then either $k > 1$
and $(2k,2k+1)\tau_{k-1}$ contains a $3$-cycle, or $k=1$ and $(2,3)\tau_2$
contains a $3$-cycle. It follows that
the transpositions in $P$ are of the form $(2k-1,2k)$ for $k \in \{1,\ldots, m\}$. Therefore
$\D(m) \le \E(m) \le P$.

By Lemma~\ref{lemma:Enormal}, $\E(m)\unlhd P$. In particular, $Q \le
\mathrm{N}_{\S_{2m}}(\E(m))$.
Since $Q\E(m)/\E(m) \cong \E(m)/(\E(m) \cap Q) = \E(m)/\D(m)$, we see that
$Q\E(m) = \langle  Q, (1,2) \rangle$ is a $2$-group having $Q$ as a subgroup of index $2$.

Let $E=E(m_1,m_2,\ldots,m_r)$ be a maximal elementary abelian $2$-subgroup of~$\S_{2m}$. Let
$F = E(m_1,m_2,\ldots,m_r) \cap \A_{2m}$. By hypothesis there exists $\sigma \in \S_{2m}$ such
that $\u\sigma F \le Q$. If $m_1=0$ then $\u\sigma E=\u\sigma F\le Q$. Suppose that $m_1\ge 1$. Let $\Omega=\{\sigma(2m_1+1),\ldots,\sigma(2m)\}$ and let $2i-1\le \sigma(1)\le 2i$. Since $\sigma(1)\not\in\Omega$, similar arguments to those used earlier in the proof show that $\{2i-1,2i\}\cap \Omega=\varnothing$. Again similar arguments show that
either
\begin{itemize} 
\item[(a)] $(\sigma(1),\sigma(2))=(2i-1,2i)$, or
\item[(b)]  $m_1=2$ and $(2i-1,2i)$ is either $(\sigma(1),\sigma(3))$ or $(\sigma(1),\sigma(4))$.
\end{itemize}
So
\[\langle (2i-1,2i),\u\sigma F\rangle=\left \{\begin{array}{ll} \u\sigma E
&\text{in case (a),} 
\\ \u\sigma(R\times E(0,m_2,\ldots,m_r))&\text{in case (b),}\end{array}\right .\] for some Sylow $2$-subgroup $R$ of $\S_4$. In either case, $\langle (2i-1,2i),\u\sigma F\rangle$ contains a $\S_{2m}$-conjugate of $E$. Clearly, $\langle (2i-1,2i),\u\sigma F\rangle\le QC(m)$. It follows that $Q\E(m)$ contains all maximal elementary abelian $2$-subgroups of $\S_{2m}$ up to $\S_{2m}$-conjugacy. By Theorem~\ref{thm:char} for $\S_n$, we see that $Q\E(m)$ is a Sylow $2$-subgroup of $\S_{2m}$. Hence $Q\E(m) \cap \A_{2m} = Q$ is a Sylow $2$-subgroup of~$\A_{2m}$.
\end{proof}


We note that this proof only used the weaker hypothesis that $Q$ contains an $\S_{2m}$-conjugate
of every elementary abelian $2$-subgroup of $\A_{2m}$. As remarked in \cite{AKMU}, the
$2$-subgroups $E(0,0,m_3,\ldots,m_r)$ are not normalized by any odd elements in $\A_{2m}$,
and so there are two $\A_{2m}$-conjugacy classes of such subgroups; in all other
cases the $\A_{2m}$- and the $\S_{2m}$-classes agree.
This may be proved as follows:\footnote{The proof in \cite{AKMU} is indicated very briefly, but seems to incorrectly assume
that if a subgroup of $\A_n$ is normalized by an odd element then it is normalized by a transposition.}
if $A = \{1,\ldots, 2^d\}$ then $\mathrm{N}_{S_{2^d}}(E_A)$
is permutation isomorphic to the affine general linear group $\mathrm{AGL}_d(\F_2)$; this group
is generated by translations and transvections, both of which
are even permutations of $\F_2^d$ when $d \ge 3$.
The normalizer of $E(0,0,m_3,\ldots,m_r)$  factors as a direct product of wreath
products $\mathrm{N}_{S_{2^d}}(E_{A_i}) \wr S_{m_i}$
where $A_i$ has size $2^i$. The permutations
in a copy of the top group $\S_{m_i}$ act on blocks of size $2^i$ and so are even.
Hence $\mathrm{N}_{S_{2^d}}(E(0,0,m_3,\ldots,m_r))
\le \A_{2m}$. It is easily seen that $\langle (12)(34) \rangle$
and $\langle (12)(34), (13)(24) \rangle$ are normalized by odd permutations in~$\S_4$,
so these
are the only exceptional cases.


\section{Proof of Theorem~\ref{thm:vertex} from Theorem~\ref{thm:elem}}\label{sec:elemtovertex}

In this section, we deduce Theorem~\ref{thm:vertex} from Theorem~\ref{thm:elem} using
Theorem~\ref{thm:char}.

Let $\overline{F}$ be the algebraic closure of $F$. By \cite[Theorem 1.21]{HuppertBlackburn}, for any $FG$-module $M$ and $H\le G$, we have that $\overline{F}\otimes_F M$ is a direct summand  
of $\overline{F}\otimes_F (M\res_H\ \ind^G)\cong (\overline{F}\otimes_F M)\res_H\ \ind^G$ if and only if $M$ is a direct summand of $M\res_H\ \ind^G$. Thus
the vertices of an indecomposable $FG$-module do not change under field extensions of $F$.

Without loss of generality, we may assume that $F$ is an algebraically closed field. Let $Q$ be a vertex of the Specht module $S^{(kp-p,1^p)}$
where $k \equiv 1$ mod $p$ and $k \not\equiv 1$ mod $p^2$.
By Theorem~\ref{thm:char}, it is sufficient to prove that if
\[ kp = m_1p + m_2p^2 + \cdots + m_rp^r \]
where $m_i \in \N_0$ for each $i \in \{1,\ldots, r\}$ and $m_r\neq 0$ then $Q$ contains a conjugate of $E(m_1,m_2,\ldots,m_r)$.
If $m_2 \not= 0$ then
Theorem~\ref{thm:elem}(ii) applies. If $m_2 = 0$ then $k = m_1 + m_3p^2 + \cdots + m_rp^{r-1}$ and
so $k \equiv m_1$ mod $p^2$. Hence $m_1 \ge 2$ and Theorem~\ref{thm:elem}(i) applies.
This completes the proof.

\section{Hook-Specht modules}
\label{sec:hook}

For maximal generality we work in this section with modules over an arbitrary
commutative ring $K$. The definition of Specht modules given in \cite[Definition 4.3]{James}
extends easily to this setting. 
Fix $n \in \N$. Let $\EE = \langle e_1, \ldots, e_n \rangle$
be the $K$-free natural permutation module for $K S_n$ of rank $n$.
Then, by definition, $S^{(n-1,1)}$ is the
$K$-free submodule of $\EE$ with $K$-basis $\{e_2-e_1,\ldots,e_n-e_1\}$.

We begin by establishing the isomorphism of $K\S_n$-modules
$S^{(n-r,1^r)} \cong \bigwedge^r S^{(n-1,1)}$. This isomorphism can be found,
most obviously in the case when $K$ is a finite field of characteristic $2$, in
Peel's papers \cite[Section 6]{PeelHooks} and~\cite{Peel75}.
When $K$ is a field of prime characteristic $p$
and $n = p$ it was proved by Hamernik \cite{Hamernik}; it is easily seen that Hamernik's
proof also works in the general case.
When $K$ is a field it was proved by M{\"u}ller and Zimmermann in
\cite[Proposition 23(a)]{MullerZimmermann}.
The proof given here combines ideas from both \cite{Hamernik} and \cite{MullerZimmermann}.

\begin{proposition}\label{prop:hookIso}
Let $1\leq r\leq n-1$.
\begin{enumerate}
\item [(i)] The set
\[ \bigl\{ (e_{i_1} - e_1) \wedge \cdots \wedge (e_{i_r} - e_1) : 1 < i_1 < \cdots < i_r \le n
\bigr\} \]
is a $K$-basis of $\bigwedge^r S^{(n-1,1)}$.

\item [(ii)] The map sending
$(e_{i_1} - e_1) \wedge \cdots \wedge (e_{i_r} - e_1) \in
\bigwedge^r S^{(n-1,1)}$ to $e_{\t} \in S^{(n-r,1^r)}$, where  $1 < i_1 < \cdots < i_r \le n$ and
$\t$ is the unique standard tableau
of shape $(n-r,1^r)$ having $i_1, \ldots, i_r$ in its rows of length one, is an isomorphism of $K\S_n$-modules.
\end{enumerate}
\end{proposition}

\begin{proof} Part (i) is obvious from the basis of $S^{(n-1,1)}$ above.
By (i) and the Standard Basis Theorem for Specht modules
(see \cite[Lemma~8.2 and Corollary~8.9]{James}),
the map defined in (ii) is a $K$-linear isomorphism.
So all we have to check is that it commutes with the action of $\S_n$.
For the subgroup of $\S_n$ fixing $1$ this is obvious.
So it suffices to check the action of the permutation $(12)$.
Let $w = (e_{i_1} - e_1) \wedge \cdots \wedge (e_{i_r} - e_1)$
where $1 < i_1 < \cdots < i_r \le n$.
If $i_1 = 2$ then it is clear that $(12)e_{\t} = -e_{\t}$ and $(12)w = -w$.
Suppose that $i_1\neq 2$. For each $1\leq a\leq r$,
let $\t_a$ be the standard tableau 
having $2, i_1, \ldots, \widehat{i_a}, \ldots, i_r$ in its rows of length $1$,
where the hat over $i_a$ indicates this entry is omitted.
By the Garnir relation (see \cite[Theorem~7.2]{James})
involving all entries in the first column of $\t$,
and the single box in the second column of $\t$,
we have
\[(12)e_{\t}=e_{\t}-e_{\t_1}+e_{\t_2}-\cdots+(-1)^re_{\t_r}.\] On the other hand,
the action on the wedge product is given by
\begin{align*}
(12)w &=
 \bigl( (e_{i_1} - e_1)-(e_2-e_1) \bigr) \wedge \cdots \wedge  \bigl( (e_{i_r} - e_1)-(e_2-e_1)
\bigr) \\
&=(e_{i_1} - e_1) \wedge \cdots \wedge (e_{i_r} - e_1)\\
&\quad\; +\sum_{a=1}^r(-1)^a(e_{2}-e_1)\wedge (e_{i_1}-e_1)\wedge\cdots\wedge\widehat{(e_{i_{a}}-e_1)}\wedge\cdots\wedge(e_{i_r}-e_1).
\end{align*}
The proof is now complete.
\end{proof}

We now introduce some further ideas from \cite{Hamernik} and basic simplicial homology.
For $r \in \N$ define
$\delta_r : \bigwedge^r \EE \rightarrow \bigwedge^{r-1} \EE$ by
\[ \delta_r (e_{i_1} \wedge \cdots \wedge e_{i_r}) =
\sum_{a=1}^r (-1)^{a-1} e_{i_1} \wedge \cdots \wedge \widehat{e_{i_a}} \wedge \cdots
e_{i_r}. \]
The subscript $r$ in the map $\delta_r$ will be omitted when it is clear from the context.
By definition $\bigwedge^0 \EE = K$. We leave it to the reader to verify the relation
\begin{equation}\label{eq:delta}
(e_{i_1} - e_j) \wedge \cdots \wedge (e_{i_r} - e_j) =
\delta(e_j \wedge e_{i_1} \wedge \cdots \wedge e_{i_r})
\end{equation}
for $1 \le j, i_1, \ldots, i_r \le n$.

\newcommand{\ru}[1]{\stackrel{\raisebox{2pt}{$\scriptstyle #1$}}{\longrightarrow}}

\begin{proposition}[Long exact sequence]\label{prop:longExact}
The sequence
\[ 0 \longrightarrow \bigwedge^n \EE \ru{\delta_n} \cdots \ru{\delta_{r+1}} \bigwedge^r \EE \ru{\delta_{r}}
\bigwedge^{r-1} \EE \ru{\delta_{r-1}} \cdots \rightarrow \EE \ru{\delta_1} K \longrightarrow 0 \]
is exact. Moreover if $1\le r < n$ then $\ker \delta_r = \im \delta_{r+1} = \bigwedge^r S^{(n-1,1)}$.
\end{proposition}

\begin{proof} The $K$-linear map defined by
$e_{i_1} \wedge \cdots \wedge e_{i_r} \mapsto e_1 \wedge e_{i_1} \wedge \cdots
\wedge e_{i_r}$ defines a homotopy equivalence between the identity map on the sequence above and the zero map. This shows that the sequence is null-homotopic and hence exact.
Proposition~\ref{prop:hookIso}(i) and Equation~\eqref{eq:delta} imply 
that $\im \delta_{r+1}$ is equal to $\bigwedge^r S^{(n-1,1)}$.
\end{proof}

We remark that there is an important homological interpretation of
Proposition~\ref{prop:longExact}.
Let $e_1, \ldots, e_n$ be the canonical basis of $\R^n$ and fix an $(n-1)$-simplex in
$\R^n$ with geometric vertices $e_1,\ldots,e_n$.
The wedge product $e_{i_1} \wedge \cdots
\wedge e_{i_r}$ can then be identified with the oriented $(r-1)$-simplex
with vertices, in order, $e_{i_1}, \ldots, e_{i_r}$.
For each $r \ge 2$,
the map $\delta_r:
\bigwedge^r \EE \rightarrow \bigwedge^{r-1} \EE$ is the boundary map from simplicial
homology, acting on oriented $(r-1)$-simplices.
Replace the map $\delta_1 : \EE \rightarrow K$
with the zero map $\EE \rightarrow 0$. Then
Proposition~\ref{prop:longExact}
is equivalent to the fundamental result
that the solid $(n-1)$-simplex has trivial homology (with coefficients in $K$)
in all non-zero degrees, and its
zero homology group is $\EE / \ker \delta_1 \cong K$.




We end this section with some further notation and results that are used in
Section~\ref{sec:proofelemii}. Let $r \in \N$. Let
\[ I^{(r)} = \{ (i_1, \ldots, i_r) : 1 \le i_1 < \cdots < i_r \le n \}. \]
We call the elements of $I^{(r)}$ \emph{multi-indices}.
Let
\[ J^{(r)} = \{ \mbf{i} \in I^{(r)} : i_1 > 1 \}. \]
For $\mbf{i} \in I^{(r)}$, 
let $e_\mbf{i}=e_{i_1}\wedge\cdots\wedge e_{i_r}$.
We say that $\{e_\mbf{i} : \mbf{i} \in I^{(r)}\}$ is the \emph{monomial basis} of $\bigwedge^r \EE$.
By \eqref{eq:delta}, the set
\begin{equation}
\label{eq:std}
\bigl\{ \delta\bigl( e_1 \wedge e_\mbf{j} \bigr) : \mbf{j} \in J^{(r)} \bigl\} 
\end{equation}
is a $K$-basis for $\bigwedge^r S^{(n-1,1)}$, corresponding under the isomorphism
in Proposition~\ref{prop:hookIso} to the standard basis of $S^{(n-r,1^r)}$.
The following lemma gives a very useful way to express elements of $\bigwedge^r S^{(n-1,1)}$
in this basis. 

\begin{lemma}[Rewriting Lemma]\label{lemma:rewrite}
Let $u = \sum_{\mbf{i} \in I^{(r)}} \mu_\mbf{i} e_\mbf{i}  \in \bigwedge^r S^{(n-1,1)}$. Then
\[ u = \sum_{\mbf{j} \in J^{(r)}} \mu_\mbf{j} \delta( e_1 \wedge e_\mbf{j} ). \]
\end{lemma}

\begin{proof}
By~\eqref{eq:std} we may write
\begin{equation*}
u  = \sum_{\mbf{j} \in J^{(r)}} \nu_\mbf{j} \delta(e_1 \wedge e_\mbf{j})
\end{equation*}
for some $\nu_\mbf{j} \in K$. 
Each of the monomial summands in $\delta(e_1 \wedge e_\mbf{j})$ involves $e_1$, with the
sole exception of the first summand $e_\mbf{j}$. Therefore
the coefficient of the monomial $e_\mbf{j}$ in the right-hand side of the above
equation is $\nu_\mbf{j}$. Hence $\nu_\mbf{j} = \mu_\mbf{j}$, as required.
\end{proof}


Finally we note that if
 $u \in \bigwedge^r \EE$ and  $v \in \bigwedge^s \EE$, where $r$, $s\in \N$, then
\begin{equation}\label{eq:deltaProd}
\delta( u \wedge v) = \delta(u) \wedge v + (-1)^r u \wedge \delta(v).
\end{equation}

%


\section{The Brauer homomorphism and monomial modules}
\label{sec:Brauer}

Throughout this section let
$F$ be a field of prime characteristic $p$ and let~$G$ be a finite group.

\subsection{The Brauer homomorphism}
Let~$V$ be an $FG$-module.
For $Q \le G$, define
\[ V^Q = \{ v \in V : \text{$\sigma v = v$ for all $\sigma \in Q$} \}. \]
For $R \leq Q \leq G$, the \emph{relative trace} map
$\Tr_{R}^{Q} : V^{R} \rightarrow V^{Q}$ is the linear map defined by
\[ \Tr_{R}^{Q}(v) = \sum_{\sigma} \sigma v \]
where the sum is over a set of coset representatives for $Q/R$.
The \emph{Brauer kernel}
of $V$ with respect to $Q$ is the subspace
$\sum_{R < Q} \Tr_R^Q V^R$
of $V^Q$.
The \emph{Brauer quotient} of $V$ with respect to $Q$
is
\[ V(Q)=V^Q / \sum_{R < Q} \Tr_R^Q V^R. \]
It is easy to show that $V^Q$ and
$\sum_{R < Q} \Tr_R^Q V^R$ are both $\mathrm{N}_{G}(Q)$-invariant, and so $V(Q)$
is a module for $F\mathrm{N}_{G}(Q)/Q$.

The following proposition
 is proved in \cite[(1.3)]{BrouePPerm}.

\begin{proposition}\label{prop:Brauer}
Let $V$ be an indecomposable $FG$-module. Let $Q$ be a $p$-subgroup of~$G$.
If $V(Q) \not= 0$ then $V$ has a vertex containing $Q$.
\end{proposition}

The converse to Proposition~\ref{prop:Brauer} is true when $V$ has the
trivial module as its source, or, equivalently, if $V$ is a direct summand of a permutation module.
In general it is false.
For example, calculations using {\sc Magma} show
that if $p=2$ and $P$ is a Sylow $2$-subgroup of $\S_{6}$
then $S^{(4,1,1)}(P) = 0$.
But by \cite[Theorems~4.3, 4.5]{MurphyPeel}, $S^{(4,1,1)}$
is indecomposable with vertex $P$.

\subsection{Monomial modules}
Let $H$ be a subgroup of $G$ and let $\theta : H \rightarrow F^\times$ be
a representation of $H$ such that $\theta(h) = 1$ whenever $h \in H$ is a $p$-element.
Let $\langle v \rangle$ be the corresponding $FH$-module.
Let $V= FG \otimes_{FH} \langle v \rangle$
be the monomial module induced from $\langle v \rangle$.
If $P$ is a Sylow $p$-subgroup of $H$ then~$\langle v \rangle$ has $P$ as a vertex, 
and so $\langle v \rangle$ is a direct summand
of $\langle v \rangle \res_P \,\, \ind^H = F\ind_P^H$.
Hence $V$ is a $p$-permutation module.
The following proposition establishes a close connection between $V$ and
the permutation module $F\ind_H^G$.

\begin{proposition}\label{prop:monomial}
Let $W = \langle e_{\gamma H} : \gamma \in G\rangle$ be the permutation module of $G$
acting on the cosets of $H$. Let $Q$ be a $p$-subgroup of $G$.
Let  $\gamma_1H, \ldots, \gamma_LH$
be representatives for the orbits of $Q$ on $G/H$.
For each $\ell$ let $s_\ell = \sum_\tau \tau \gamma_\ell \otimes v$
where the sum is over a set of coset representatives $\tau$ for
$Q/\Stab_Q (\gamma_\ell H)$.
Then

\begin{enumerate}
\item [(i)] the map defined by $\alpha \gamma_\ell \otimes v \mapsto \alpha e_{\gamma_\ell H}$ for $\alpha \in Q$ is
an isomorphism $V \res_Q\; \cong W\res_Q$,

\item [(ii)] $s_\ell = \Tr_{\Stab_Q (\gamma_\ell H)}^Q (\gamma_\ell \otimes v)$,

\item [(iii)] $\{s_\ell : 1 \le \ell \le L\}$ is a basis for $V^Q$,

\item [(iv)] if $R \le Q$ then $\{s_\ell : 1 \le \ell \le L, \Stab_Q (\gamma_\ell H)
= \Stab_R (\gamma_\ell H) \}$ is a basis for $\Tr_R^Q V^R$.
\end{enumerate}
\end{proposition}

\begin{proof}
Part (i) follows from Mackey's induction/restriction formula on noting
that if $\gamma \in G$ then
the restriction of $\u\gamma \theta$ to $\u \gamma H\cap Q$
is trivial.
The orbit sums of $Q$ acting on the canonical permutation
basis of $W$ are a basis for~$W^Q$. Since the orbit sum containing $e_{\gamma_\ell H}$ is
$\Tr_{\Stab_Q (\gamma_\ell H)}^Q e_{\gamma_\ell H}$, parts (ii) and (iii) follow from (i).
For (iv), let $\gamma \in G$ and note that by (iii) the orbit sum of~$R$
containing~$e_{\gamma H}$ is $w = \Tr_{\Stab_R (\gamma H)}^R e_{\gamma H}$.
If $\Stab_R (\gamma H) = \Stab_Q (\gamma H)$ then
$\Tr_R^Q w = \Tr_{\Stab_Q(\gamma H)}^Q e_{\gamma H}$ is the orbit sum of $Q$ containing
$e_{\gamma H}$. On the other hand,
if $\Stab_R (\gamma H) < \Stab_Q (\gamma H)$ then
\[ \Tr_R^Q w = \Tr_{\Stab_Q(\gamma H)}^Q \Tr_{\Stab_R(\gamma H)}^{\Stab_Q(\gamma H)} e_{\gamma H}
           = \Tr_{\Stab_Q(\gamma H)}^Q 0 = 0.\]
Part (iv) now follows from the isomorphism in (i).
\end{proof}

\section{Generic Jordan types of modules}
\label{sec:gJt}

Let $E=\langle g_1,\ldots,g_n\rangle$ be an elementary abelian $p$-group of $p$-rank $n$ and
let $F$ be an algebraically closed field of characteristic $p$.
Let $M$ be a finite-dimensional $FE$-module. Let $K = F(\alpha_1,\ldots,\alpha_n)$
where $\alpha_1, \ldots,\alpha_n$ are indeterminates.  
With respect to a basis for $M$,
the matrix representing the action of the element
\[1+\alpha_1(g_1-1)+\cdots+\alpha_n(g_n-1)\]
on $M$ has order $p$.
If $[r]$ denotes a unipotent Jordan block of dimension $r$ then the
Jordan type of this matrix is $[1]^{s_1}\cdots [p]^{s_p}$ for some $s_i \in \N_0$.
By Wheeler \cite{Wheeler},
this Jordan type is independent of the choice of the generators of $E$.
It is called the \emph{generic Jordan type} of the $FE$-module $M$. The \emph{stable generic Jordan type} of $M$ is
$[1]^{s_1}\cdots[p-1]^{s_{p-1}}$.
The module $M$ is \emph{generically free}
if $s_1=\cdots=s_{p-1}=0$. For further background on  generic Jordan type,
we refer the reader to \cite{Friedlander-Pevtsova-Suslin}.

We summarize below the properties which we need.

\begin{proposition}\label{prop:sgJt} Suppose that $E$ is an elementary abelian $p$-group of finite order.
\begin{enumerate}
  \item [(i)]
  The generic Jordan type of a direct sum of modules
  is the direct sum of the generic Jordan types of the modules.
  \item [(ii)] Let $0\to M_1\to M_2\to M_3\to 0$ be a short exact sequence of $FE$-modules.
  \begin{enumerate}
  \item [(a)] If $M_2$ is generically free then $M_1$ has stable generic Jordan type $[1]^{s_1}\cdots[p-1]^{s_{p-1}}$ if and only if $M_3$ has stable generic Jordan type $[1]^{s_{p-1}}\cdots[p-1]^{s_{1}}$.
  \item [(b)] If $M_3$ is generically free then $M_1$
  and $M_2$ have the same stable generic Jordan type.
  \item [(c)] If $M_1$ is generically free then $M_3$ and $M_2$ have the same
   stable generic Jordan type. 

  \end{enumerate}
  \item [(iii)] Let $D$ be a proper subgroup of $E$ and let $U$ be a $FD$-module. Then the induced module $U\ind^E$ is generically free.
  \item [(iv)] Let $G$ be a finite, not necessarily elementary abelian, supergroup of~$E$
  and let $V$ be an indecomposable $FG$-module.
  If $V\res_E$ is not generically free then $V$ has a vertex containing the subgroup $E$.
\end{enumerate}
\end{proposition}
\begin{proof} Part (i) follows from \cite[Proposition 4.7]{Friedlander-Pevtsova-Suslin}
and will be used in the proofs of (ii), (iii) and (iv).
Since we do not require these ideas elsewhere in our paper,
we refer the reader to \cite{Friedlander-Pevtsova-Suslin} for details.
We write $[\alpha_K](M)$ and
$[\alpha_K]^\star(M)$ for the generic and stable generic Jordan type of a module~$M$ respectively.

For part (ii), we note that the short exact sequence induces the short exact sequence \[0\to [\alpha_K](M_1)\to [\alpha_K](M_2)\to [\alpha_K](M_3)\to 0.\] For (ii)(a) we have
\[
[\alpha_K]^\star(M_1)\simeq \Omega\bigl( \Omega^{-1}[\alpha_K]^\star(M_1) \bigr)
\simeq\Omega\bigl( [\alpha_K]^\star(M_3) \bigr).\]
Hence the stable generic Jordan types of $M_1$ and $M_3$ are complementary.
For (ii)(b) we note that since
$M_3$ is generically free, the short exact sequence generically splits.
Hence
\[
[\alpha_K]^\star(M_2)\simeq [\alpha_K]^\star(M_1)\oplus [\alpha_K]^*(M_3)=[\alpha_K]^*(M_1).
\]
The proof of (ii)(c) is similar.


There exists a non-trivial subgroup $C$ of $E$ such that $E = C \times D$. We have
$U\ind^E \cong FE \otimes_{FD} U \cong F(E/D) \otimes_F U'$ where $U'$ has
the same underlying vector space as $U$ and the action
of $E$ is given by $(hk)v = kv$ for $h \in C$,
$k \in D$ and $v \in U'$. Since $F(E/D)$ is generically free
as an $FE$-module, it follows that $U\ind^E$ is generically free, as required in~(iii).


For (iv), suppose that $V$ has $Q$ as a vertex.
Then there is an $FQ$-module~$M$ such that $V$ is a direct summand of $M\ind^{G}$.
 By the Mackey decomposition formula $(M\ind^{G})\res_E$ is a summand of a direct sum of
the modules $M_g=\left (\u{g}M\res_{E\cap gQg^{-1}}\right )\ind^E$ for some suitable $g\in G$.
If $E\cap gQg^{-1}$ is a proper subgroup of $E$ then $M_g$ is
generically free by (iii). Since $V\res_E$ is direct summand of $(M\ind^G)\res_E$ and,
by hypothesis $V\res_E$ is not generically free, there
exists $g\in G$ such that $E\cap gQg^{-1}=E$. For this $g$ we have $E\subseteq gQg^{-1}$.
\end{proof}

Of independent interest, we mention a direct consequence of Proposition~\ref{prop:sgJt}(iv) about hook Specht modules. Let $n \in \N$ and let $d = \lfloor \frac{n}{p} \rfloor$. Let
\[ \E(d)
= \langle (1,2,\ldots,p),\ldots,((d-1)p+1,(d-1)p+2,\ldots,dp) \rangle
\]
be as defined
in Section~\ref{sec:charS}. The second author
computed the generic Jordan type of
$S^{(n-r,1^r)}\res_{\E(d)}$ in
\cite[Corollary 4.2, Theorem 4.5]{LimComplexity}.
In particular, the restricted module is not generically free. Hence by
Proposition~\ref{prop:sgJt}(iv), we obtain the following result.

\begin{proposition} The hook Specht module $S^{(n-r,1^r)}$ has a vertex containing the maximal elementary abelian $p$-subgroup $\E(d)$
where
$d=\lfloor \frac{n}{p}\rfloor$.
\end{proposition}

We end this section
with a result on the generic Jordan type of a monomial module.

\begin{proposition}\label{prop:gJt of monomial mod} Let $G$
be a finite group, let $H \le G$ and let $V=\langle v \rangle \ind_H^G$
be a monomial $FG$-module. If $E$ is an elementary abelian $p$-subgroup of $G$
then $V\res_E$ has generic Jordan type $[1]^s$ where $s$ is the number of
orbits of size $1$ of $E$ acting on $G/H$.
\end{proposition}

\begin{proof}
We have
\[V\Res_E=\bigoplus_{g} (\u{g}\langle v\rangle \Res_{E\hskip1pt\cap\hskip1pt gHg^{-1}})\Ind^E \]
where the sum is over a set of representatives for the orbits of $E$ on $G/H$.
Orbits of size $1$ correspond to
representatives $g$ such that  $E\subseteq gHg^{-1}$.  
In this case,
\[(\u{g}\langle v\rangle \Res_{E\hskip1pt\cap\hskip1pt gHg^{-1}})
\Ind^E=\u{g}\langle v\rangle \Res_E.
\]
Since $\u{g}\langle v\rangle\res_E$ is $1$-dimensional, its generic Jordan type is clearly $[1]$. The result now follows from parts~(i) and~(iii) of Proposition \ref{prop:sgJt}.
\end{proof}

\section{Proof of Theorem~\ref{thm:elem} under hypothesis (i)}
\label{sec:proofelemi}

Let $F$ be an algebraically closed field of prime characteristic $p$.
As in Section~\ref{sec:hook} we let
$\EE = \langle v_1, \ldots, v_{kp}\rangle$
be the $kp$-dimensional
natural permutation module for $\S_{kp}$. Recall from the end of Section~\ref{sec:hook}
that $\bigwedge^r \EE$ has as a basis the elements
$e_{\mathbf{i}} = e_{i_1}\wedge \cdots\wedge e_{i_r}$
for $\mathbf{i}\in I^{(r)}$. Considering the action
of $H = \S_{\{1,\ldots, r\}} \times \S_{\{r+1,\ldots, kp\}}$
on the generator $e_1 \wedge \cdots \wedge e_r$ of $\bigwedge^r \EE$,
we see that
\begin{equation}\label{eqn:monomial}
\bigwedge^r \EE 
 \cong \langle v\rangle\Ind^{\S_n}
\end{equation}
where $\langle v \rangle$ affords the representation $\sgn \boxtimes F$ of $H$.
Hence $\bigwedge^r \EE$ is a monomial module.



\begin{corollary}\label{Cor:sgJt of wedge E} Suppose that $n=kp$. Let $E$
be an elementary abelian $p$-subgroup of $\S_{kp}$.  
Then the stable generic Jordan type of $\left (\bigwedge^r \EE\right )\res_E$ is $[1]^s$ where $s$ is the number of $r$-subsets of $\{1,\ldots,kp\}$ fixed by $E$ in its
action on all $r$-subsets of $\{1,\ldots,kp\}$.
 In particular, $\left (\bigwedge^r \EE\right )\res_E$ is generically free if $r\not\equiv 0$ mod $p$.
\end{corollary}

\begin{proof} Since $\bigwedge^r\EE$ is a monomial module by \eqref{eqn:monomial}, we may apply
Proposition~\ref{prop:gJt of monomial mod}. Identifying $\S_n / H$ with the set of $r$-subsets
of $\{1,\ldots, kp\}$ we see that orbits of $E$ on $\S_n / H$ of size $1$ correspond
to $r$-subsets of $\{1,\ldots, kp\}$ fixed by $E$. Hence
the generic Jordan type of $(\bigwedge^r\EE)\res_E$ is $[1]^s$ where $s$ is the
number of $r$-subsets of $\{1,\ldots,kp\}$ which are fixed under the action of $E$.
Notice that all the orbits of $E$ on $\{1,\ldots,kp\}$ have sizes $p^j$ for some $j\geq 1$. Thus $s=0$ if $r$ is not divisible by $p$.
\end{proof}

We are now ready to prove Theorem~\ref{thm:elem} under hypothesis (i).
Let $E=E(m_1,m_2,\ldots, m_t)$ where $m_1\geq 2$.  
By Proposition \ref{prop:sgJt}(iv),
it suffices to show that  $S^{(kp-p,1^p)}\res_E\; \cong \bigwedge^p S^{(kp-1,1)} \res_E$ is not
generically free. We truncate the long exact sequence in
Proposition \ref{prop:longExact} to obtain
\[ 0 \longrightarrow \bigwedge^p S^{(kp-1,1)} \longrightarrow \bigwedge^p \EE \ru{\delta_p}
\bigwedge^{p-1} \EE \ru{\delta_{p-1}} \cdots \rightarrow \EE \ru{\delta_1} F \longrightarrow 0.\]
Restrict each term in the long exact sequence to  $E$. For each $1\leq i\leq p$, we have a short exact sequence
\[0\longrightarrow 
\bigwedge^i S^{(kp-1,1)}
\Res_E\longrightarrow 
\bigwedge^i \EE 
\Res_E\ru{\delta_{i}} 
\bigwedge^{i-1} S^{(kp-1,1)} 
\Res_E\longrightarrow 0. \]
By Corollary \ref{Cor:sgJt of wedge E},
$\bigwedge^i \EE \res_E$
is generically free if $1 \le i \le p-1$.
Since $F\res_E$ 
has generic Jordan type $[1]$, it follows from
Proposition \ref{prop:sgJt}(ii)(a)
that
$S^{(kp-1,1)} \res_E$
has stable generic Jordan type $[p-1]$. Repeating this argument, with
$\bigwedge^{2} S^{(kp-1,1)} \res_E$ up to $\bigwedge^{p-1} S^{(kp-1,1)}\res_E$,
we find that $\bigwedge^{p-1} S^{(kp-1,1)}$ has stable generic Jordan type $[1]$.

Suppose for a contradiction that  $\bigwedge^{p} S^{(kp-1,1)}\res_E$
is generically free. Then, by Proposition \ref{prop:sgJt}(ii)(c), 
$\bigwedge^p \EE \res_E$ has stable generic Jordan type $[1]$.
On the other hand, since $E=E(m_1,m_2,\ldots,m_t)$ where $m_1\geq 2$, there are
exactly $m_1$ orbits of $E$ on $\{1,\ldots,kp\}$ of size $p$. Each such 
orbit corresponds to a $p$-subset of $\{1,\ldots,kp\}$ fixed by $E$ so, by
Corollary~\ref{Cor:sgJt of wedge E},
$\bigwedge^p \EE \res_E$ has
stable generic Jordan type $[1]^{m_1}$. The contradiction shows that
$S^{(kp-p,1^p)} \res_E$ is not generically free. This completes the proof.

\section{Proof of Theorem~\ref{thm:elem} under hypothesis (ii)}
\label{sec:proofelemii}

Let $F$ be a field of prime characteristic $p$. Under the hypothesis (ii) of Theorem~\ref{thm:elem}, we have that $k\geq p$. As in Section~\ref{sec:hook}, we let $\EE = \langle v_1, \ldots, v_{kp} \rangle$
be the $kp$-dimensional
natural permutation module for $\S_{kp}$ and define $S^{(kp-1,1)}$
to be the submodule with $F$-basis $\{e_2-e_1,\ldots, e_{kp}-e_1\}$.
Let $W = \bigwedge^p S^{(kp-1,1)}$.  By Proposition~\ref{prop:hookIso}
we have $W \cong S^{(kp-p,1^p)}$.  

Let $T$ be a Sylow $p$-subgroup of the symmetric group on $\{p^2+1,\ldots, kp\}$.
Let $Q$ be the subgroup of $\S_{kp}$ generated by $T$ together with
\begin{align*}
\alpha &= (1,2,\ldots, p) \cdots (p^2-p+1,p^2-p+2,\ldots,p^2), \\
\beta  &= (1,p+1,\ldots, p^2-p+1)\cdots (p,2p,\ldots, p^2),
\end{align*} and let
\begin{equation*}
w = \bigl( e_1 + e_{p+1} + \cdots + e_{(p-1)p+1} \bigr) \wedge \cdots \wedge
\bigl( e_p + e_{2p} + \cdots + e_{p^2} \bigr).
\end{equation*}
Since each factor in the wedge product lies in
$S^{(kp-1,1)}$, and $w$ is fixed by $\alpha$,~$\beta$ and $T$,
we have $w \in W^Q$.  Observe that the coefficient of $e_1 \wedge e_{2} \wedge \cdots
\wedge e_{p}$
in $w$ is $1$.

The main result of this section is Proposition~\ref{prop:small_elem} below. It gives us the statement of Theorem~\ref{thm:elem}(ii) almost immediately.

\begin{proposition}\label{prop:small_elem}
If $R < Q$ then no element of $\Tr_R^Q W^R$
has a non-zero coefficient of
$e_1  \wedge e_{2} \cdots \wedge e_{p}$
 when expressed in the monomial basis
of~$\bigwedge^p \EE$.
\end{proposition}

\begin{proof}[Proof of Theorem~\ref{thm:elem}(ii) assuming Proposition~\ref{prop:small_elem}]
Let $E$ be an elementary abelian $p$-subgroup of $\S_{kp}$ having at least one
orbit of size $p^2$. There exists $\sigma \in \S_{kp}$ such that
\[ \sigma E \sigma^{-1} \le \langle \alpha, \beta \rangle \times T = Q. \]
By Proposition~\ref{prop:small_elem}, we have $W(Q) \not= 0$.
Theorem~\ref{thm:elem}(ii) now follows from
Proposition~\ref{prop:Brauer}.
\end{proof}

The proof of Proposition \ref{prop:small_elem} occupies the remainder of this section.
It is divided into two steps: in the first
we reduce the proposition
to a question about a specific
relative trace map on the smaller module $\bigwedge^p S^{(p^2-1,1)}$.
In the second step we answer this question (see Proposition~\ref{prop:subspace_induction})
using a carefully chosen filtration of $\bigwedge^p S^{(p^2-1,1)}$.

\subsection*{Step 1: Reduction} We need the following lemma.

\begin{lemma}\label{lemma:E_basis}
Let $R \le Q$, let $B^{(1)}, \ldots, B^{(m)}$
be representatives for the orbits of $R$ on
the set of $p$-subsets of $\{1,\ldots,kp\}$
and let $\mathbf{i}^{(1)}, \ldots, \mathbf{i}^{(m)}$
be the corresponding multi-indices. Then a basis for
$(\bigwedge^p \EE)^R$ is
\[ \mathcal{B} = \{ \Tr_{\Stab_R (B^{(k)})}^R e_{\mathbf{i}^{(k)}} : 1 \le k \le m \}. \]
\end{lemma}

\begin{proof}
By (6) at the beginning of Section \ref{sec:proofelemi}, the module $\bigwedge^p \EE$ is induced from the representation $\sgn \boxtimes
F$
of $H = \S_{\{1,\ldots, p\}} \times \S_{\{p+1,\ldots, kp\}}$.
Since $\sgn(g) = 1$ for all $p$-elements $g \in \S_p$,
the hypotheses of Proposition~\ref{prop:monomial} are satisfied.
The lemma now follows from parts (ii) and~(iii) of this proposition,
noting as in the proof of Corollary~\ref{Cor:sgJt of wedge E}
that the coset space $\S_{kp}/H$ is
isomorphic as an $\S_{kp}$-set
to the set of all $p$-subsets of $\{1,\ldots, kp\}$. 
\end{proof}

We now show that only one subgroup $R$ needs to be considered in Proposition~\ref{prop:small_elem}.

\begin{proposition}\label{prop:orbit_reduction}
If $R$ is a proper subgroup of $Q$ such that
$e_1 \wedge e_{2} \wedge \cdots \wedge e_{p}$
appears with a non-zero coefficient in an element of $\Tr_R^Q W^R$, then
$R = \langle \alpha, T \rangle$.
\end{proposition}

\begin{proof}
Since $W^R$ is a submodule of $(\bigwedge^p \EE)^R$,
it is sufficient to prove the proposition with $W$ replaced by $\bigwedge^p \EE$.
Let $\mathcal{B}$ be the basis of $(\bigwedge^p \EE)^Q$ given by Lemma~\ref{lemma:E_basis}.
The unique element of $\mathcal{B}$
containing $e_1 \wedge e_{2} \wedge \cdots \wedge e_{p}$
with a non-zero coefficient is
\[ \Tr_{\langle \alpha, T \rangle}^{Q} (e_1 \wedge e_{2} \wedge \cdots \wedge e_{p})
= e_1 \wedge e_{2} \wedge \cdots \wedge e_{p} + \cdots +
       e_{(p-1)p+1} \wedge e_{(p-1)p+2} \wedge \cdots \wedge e_{p^2}. \]
By Proposition~\ref{prop:monomial}(iv),
$\Tr_R^Q (\bigwedge^p \EE)^R$ has, as a basis, a subset $\mathcal{B}'$ of $\mathcal{B}$
such that $\mathcal{B}'$ contains
$\Tr_{\langle \alpha, T \rangle}^{Q} (e_1 \wedge e_{2} \wedge \cdots \wedge e_{p})$
if and only if
\[ \Stab_Q (\{1, 2, \ldots, p\}) = \Stab_R (\{1,2,\ldots,p\}). \]
Since $\Stab_Q (\{1, 2, \ldots, p\}) = \langle \alpha , T \rangle$ is a maximal
subgroup of $Q$, this condition holds if and only if
$R = \langle \alpha, T \rangle$. 
\end{proof}

The next proposition completes the reduction step.

\begin{proposition}\label{prop:subspace_reduction}
If $e_1 \wedge e_{2} \wedge \cdots \wedge e_{p}$
appears with a non-zero coefficient in
$\Tr_{\langle \alpha, T \rangle}^Q v$ for some $v \in (\bigwedge^p S^{(kp-1,1)})^{\langle\alpha,T\rangle}$
then $e_1 \wedge e_{2} \wedge \cdots \wedge e_{p}$
appears with a non-zero coefficient in
$\Tr_{\langle \alpha \rangle}^{\langle \alpha, \beta \rangle}
v'$ for some $v' \in (\bigwedge^p S^{(p^2-1,1)})^{\langle \alpha\rangle}$.
\end{proposition}

\begin{proof}
Let $B^{(1)}, \ldots, B^{(L)}$ be representatives for the orbits of $\langle \alpha \rangle$
on the set of $p$-subsets of
$\{1,\ldots, p^2\}$ and let $D^{(1)}, \ldots, D^{(M)}$ be representatives
for the orbits of $\langle \alpha, T \rangle$ on the set of $p$-subsets of
$\{1,\ldots,kp\}$
that have non-empty intersection with $\{p^2+1,\ldots, kp\}$.
Let $\mathbf{i}^{(1)}, \ldots, \mathbf{i}^{(L)}$ and $\mathbf{k}^{(1)}, \ldots, \mathbf{k}^{(M)}$
be the corresponding multi-indices, and let
\begin{align*}
s_\ell &= \Tr_{\Stab_{\langle \alpha, T \rangle} (B^{(\ell)})}^{\langle \alpha, T \rangle}
e_{\mathbf{i}^{(\ell)}}
\quad \text{for $1 \le \ell \le L$,}\\
t_m &= \Tr_{\Stab_{\langle \alpha, T \rangle} (D^{(m)})}^{\langle \alpha, T \rangle}
e_{\mathbf{k}^{(m)}}
\quad \text{for $1 \le m \le M$.}
\end{align*}
By Lemma~\ref{lemma:E_basis}, 
$\{ s_\ell : 1 \le \ell \le L \}$ is a
basis for $(\bigwedge^p \langle e_1, \ldots e_{p^2} \rangle)^{\langle \alpha \rangle}$ and
$\{ s_\ell : 1 \le \ell \le L\} \cup \{t_m : 1 \le m \le M\}$
is a basis for $(\bigwedge^p \EE)^{\langle \alpha, T \rangle}$.

Let $v = \sum_{\ell=1}^L \lambda_\ell s_\ell +
\sum_{m=1}^M \mu_m t_m$ where the coefficients $\lambda_\ell$ and $\mu_m$ are in $F$.
By Proposition~\ref{prop:longExact}, $\bigwedge^p S^{(kp-1,1)}$ is the kernel
of the map $\delta : \bigwedge^p \EE \rightarrow \bigwedge^{p-1} \EE$.
There is a vector space decomposition $\bigwedge^{p-1} \EE = U \oplus U'$ where
\begin{align*}
U  &= \langle e_{\mathbf{i}} : \mbf{i} \in I^{(p-1)}, 1 \le i_1 < \ldots < i_{p-1} \le p^2 \rangle, \\
U' &= \langle e_{\mathbf{i}} : \mbf{i} \in I^{(p-1)}, i_{p-1} > p^2 \rangle.
\end{align*}
For each $s_\ell$ we have $\delta(s_\ell) \in U$. Let $1 \le m \le M$.
If $|D^{(m)} \cap \{p^2+1,\ldots, kp\}| \ge 2$ then every monomial summand
of $t_m$ involves two or more of $e_{p^2+1}, \ldots, e_{kp}$, and so
$\delta(t_m)$ lies in $U'$. Suppose that $D^{(m)} = \{i_1, \ldots, i_{p-1}, i_p\}$
where $1 \le i_1 \le \ldots \le i_{p-1} \le p^2$ and $i_p > p^2$.
Since $\Stab_{\langle \alpha, T\rangle} (\{ i_1,\ldots, i_{p-1}\}) = T$,
we have
\begin{align*}
t_m &= \Tr^{\langle \alpha, T \rangle}_{\Stab_T (\{i_p\})}
\bigl( (e_{i_1} \wedge \cdots \wedge e_{i_{p-1}}) \wedge e_{i_p} \bigr) \\
    &= \Bigl( \Tr_{1}^{\langle \alpha \rangle} (e_{i_1} \wedge \cdots \wedge e_{i_{p-1}}) \Bigr)
    \wedge \Tr_{\Stab_T (\{i_p\})}^T e_{i_p}.
\end{align*}
Since the orbit of $T$ containing $i_p$ has size a multiple of $p$,
it now follows from \eqref{eq:deltaProd} at the end of Section~\ref{sec:hook} that
\[ \delta(t_m) = \delta \Bigl( \Tr_1^{\langle \alpha \rangle}
(e_{i_1} \wedge \cdots \wedge e_{i_{p-1}}) \Bigr) \wedge  \Tr_{\Stab_T (\{i_p\})}^T e_{i_p}.\]
Hence $\delta(t_m) \in U'$. Thus
$\delta(\sum_{\ell=1}^L \lambda_\ell s_\ell) \in U$,
whereas $\delta(\sum_{m=1}^M \mu_m t_m) \in U'$.
Since $\delta(v) = 0$ and $U \cap U' = \{ 0 \}$, it follows that
$\delta(\sum_{\ell=1}^L \lambda_\ell s_\ell)  = 0$. By Proposition~\ref{prop:longExact},
$\sum_{\ell=1}^L \lambda_\ell s_\ell \in \bigwedge^p (S^{(p^2-1,1)})$.
It is clear that no element of $\Tr_{\langle \alpha, T \rangle}^Q \langle t_1,\ldots,t_M\rangle$ contains the
monomial $e_{1} \wedge e_2 \wedge \ldots \wedge e_{p}$ with a non-zero coefficient.
Therefore $v' = \sum_{\ell=1}^L \lambda_\ell s_\ell$  has the properties
required in the proposition.
\end{proof}

\renewcommand{\alph}{\alpha} 

\subsection*{Step 2: Proof of the reduced version of Proposition~\ref{prop:small_elem}}
Fix $u \in (\bigwedge^p S^{(p^2-1,1)})^{\langle \alph \rangle}$.
By Propositions~\ref{prop:orbit_reduction} and~\ref{prop:subspace_reduction}, to
prove Proposition~\ref{prop:small_elem},
it suffices to show that the coefficient of $e_{1} \wedge e_2 \wedge \cdots \wedge e_{p}$ in
$\Tr_{\langle \alph \rangle}^{\langle \alph, \beta \rangle} u$ is zero.
We do this by analysing the behaviour of this trace map
on a filtration of $(\bigwedge^p S^{(p^2-1,1)})^{\langle \alph \rangle}$. 
From now on we work inside  $\bigwedge^p \langle e_1, \ldots, e_{p^2}\rangle$.
To simplify the notation, let $V = \bigwedge^p S^{(p^2-1,1)}$, and let
$M = \bigwedge^p \langle e_1, \ldots, e_{p^2}\rangle$.
Let $I = I^{(p)}$ and let $J = J^{(p)}$, where these
sets of multi-indices are as defined in Section~\ref{sec:hook} when $n=p^2$. (Thus
$I = \{ (i_1,\ldots, i_p) : 1 \le i_1 < \cdots <i_p \le p^2 \}$ and $J = \{ \mathbf{i} \in I:
i_1 > 1\}$.)

Let $\Delta = \{1,2,\ldots,p\}$.
For $0 \le c \le p$ let
\[
I_c = \{ \mathbf{i} \in I : i_1, \ldots, i_c \in \Delta, i_{c+1}, \ldots, i_p \not\in \Delta \}
\]
and let $M_c = \langle e_\mathbf{i} : \mathbf{i} \in I_c \rangle$. Since $\Delta$
is an orbit of $\alpha$, the subspaces
$M_c$ are invariant under $\alpha$, so as an $F \langle \alpha \rangle$-module we have
$M = M_0 \oplus M_1 \oplus \cdots \oplus M_p$.
For $0 \le c \le p$ let
\[ V_c = V \cap (M_c \oplus \cdots \oplus M_p). \]
Since $M_p$ is spanned by 
$e_1 \wedge e_2 \wedge \cdots \wedge e_p$
and $\delta(e_1\wedge e_2 \wedge \cdots \wedge e_p) \not= 0$,
we have $V_p = \{0\}$.
The required result that the coefficient of $e_1 \wedge e_{2} \wedge \cdots \wedge e_{p}$ in $\Tr_{\langle \alpha\rangle}^{\langle\alpha,\beta\rangle}u$
is zero therefore follows immediately from
the case $c = p-1$ of
the next proposition.

\begin{proposition}\label{prop:subspace_induction}
Let  $0 \le c \le p-1$. There exists $v_0 \in V_0^{\langle \alph \rangle}$,
\ldots $v_c \in V_c^{\langle \alph \rangle}$
such that
\begin{enumerate}
\item [(i)] the coefficient of $e_{1} \wedge \cdots \wedge e_{p}$ in
$\Tr_{\langle \alph\rangle}^{\langle \alph, \beta \rangle} (v_0 + \cdots + v_c) $ is zero,
\item [(ii)] $u \in v_0 + \cdots + v_c + V_{c+1}$.
\end{enumerate}
\end{proposition}

We prove Proposition~\ref{prop:subspace_induction} by induction on $c$, using
the following two lemmas.

\begin{lemma}\label{lemma:zm}
Fix $2 \le m \le p$.
For each $X \subseteq \{1,\ldots, p\}$ we define
a $p$-tuple $\mathbf{i}(X)$ with elements in $\{1,\dots,p\} \cup \{(m-1)p+1,\ldots,(m-1)p+p\}$ by
\[ \mathbf{i}(X)_a = \begin{cases} a & \text{if $a \in X$,} \\
(m-1)p + a  & \text{if $a \not\in X$.} \end{cases}
\]
Define
\[ z_m = \sum_{X \subseteq \{1, \ldots, p\}} (-1)^{|X|} e_{\mathbf{i}(X)}. \]
Then
\begin{enumerate}
\item [(i)] $z_m \in V^{\langle \alph \rangle}$,
\item [(ii)] $z_m \in e_{(m-1)p+1} \wedge \cdots \wedge e_{(m-1)p + p}
+ M_1 \oplus \cdots \oplus M_p$,
\item [(iii)] the coefficient of $e_{1} \wedge e_{2} \wedge \cdots \wedge e_{p}$ in
$\Tr_{\langle \alph \rangle}^{\langle \alph, \beta \rangle} z_m$ is zero.
\end{enumerate}
\end{lemma}

\begin{proof}
The monomial summands in
\[ \delta(z_m) = \sum_{X \subseteq \{1, \ldots, p\}} (-1)^{|X|}
\sum_{a=1}^p (-1)^{a-1} \, e_{\mathbf{i}(X)_1} \wedge \ldots \wedge
\widehat{e_{\mathbf{i}(X)_a}} \wedge \ldots \wedge e_{\mathbf{i}(X)_p}\]
are indexed by pairs $(X, a)$ where $X \subseteq \{1,\ldots, p\}$
and $1 \le a \le p$. We define an involution $\star$ on the set of such pairs by
\[ (X,a)^\star = (X \symdiff \{a\},a). \]
where $\symdiff$ denotes symmetric difference. The summands for
$(X,a)$ and $(X,a)^\star$ cancel. Hence $\delta(z_m) = 0$ and $z_m \in V$.
Since $\alph e_{\mathbf{i}(X)} = e_{\mathbf{i} (\alph X)}$ and obviously $|\alph X|=|X|$,
we have $z_m \in V^{\langle \alph \rangle}$.
The unique monomial summand of~$z_m$ in $M_0$ is $e_{\mathbf{i}(\emptyset)}
= e_{(m-1)p+1} \wedge \cdots \wedge e_{(m-1)p + p}$, so we have (ii). The two
summands of~$z_m$ that contribute to the coefficient of
$e_{1} \wedge e_2 \wedge \cdots \wedge e_{p}$ in
$\Tr_{\langle \alph \rangle}^{\langle \alph, \beta \rangle}z_m$
are $e_{\mathbf{i}(\emptyset)}$, which appears with coefficient $+1$,
and $e_{\mathbf{i}(\{1,2,\ldots,p\})} = e_1 \wedge e_2 \wedge \cdots \wedge e_p$, which appears with coefficient $-1$.
Their contributions cancel, hence (iii).
\end{proof}

\begin{lemma}\label{lemma:wj}
Let $0\le c \le p-1$ and let $\mathbf{j} = (j_1,\ldots,j_p) \in J \cap I_c$.
Define
\[ w(\mathbf{j}) = \sum_{\ell=0}^{p-1} \alph^\ell \delta(e_1 \wedge e_\mathbf{j}). \]
Then
\begin{enumerate}
\item [(i)] $w(\mathbf{j}) \in V_c^{\langle \alph \rangle}$,
\item [(ii)] $w(\mathbf{j}) \in \sum_{\ell=0}^{p-1} \alph^\ell \bigl( \delta(e_1 \wedge e_\mathbf{d}) \wedge
e_\mathbf{k} \bigr) + M_{c+1}$,
where $\mathbf{d} = (j_1, \ldots, j_c)$ and $\mathbf{k} = (j_{c+1}, \ldots, j_p)$.
\item [(iii)] the coefficient of $e_{1} \wedge e_{2} \wedge \cdots \wedge e_{p}$ in
$\Tr_{\langle \alph \rangle}^{\langle \alph, \beta \rangle} w(\mathbf{j})$
is zero.

\end{enumerate}
\end{lemma}

\begin{proof} Clearly $w(\mbf{j})\in V^{\langle \alpha\rangle}$.
The monomial summands of
\[ \alph^\ell \delta(e_1 \wedge e_\mathbf{j}) = \delta(e_{1+\ell} \wedge e_{\alph^\ell (j_1)} \wedge
\cdots \wedge e_{\alpha^\ell (j_c)} \wedge e_{\alpha^{\ell} (j_{c+1})} \wedge \cdots
\wedge e_{\alph^\ell (j_p)}) \]
are $e_{\alph^\ell (j_1)} \wedge \cdots \wedge e_{\alph^\ell (j_p)}$ and
$e_{1+\ell} \wedge e_{\alph^\ell (j_1)} \wedge \cdots \wedge \widehat{e_{\alph^\ell (j_a)}} \wedge \cdots\wedge
e_{\alph^\ell (j_p)}$ for~$1 \le a \le p$. The first summand and the summands for $a$
such that~$1 \le a \le c$ are in $M_c$, and the rest are in $M_{c+1}$, hence (i) and (ii).
If $e_\mathbf{i} \in M_b$ where $1 \le b \le p-1$ then the coefficient
of $e_{1} \wedge \ldots \wedge e_{p}$ in
$\Tr_{\langle \alph \rangle}^{\langle \alpha, \beta \rangle} e_\mathbf{i}$ is zero.
This implies that (iii) holds when $1 \le c \le p-2$.

Suppose that $c=0$. Then the
only contributions to the coefficient of $e_{1} \wedge e_2 \wedge \cdots \wedge e_{p}$
in $\Tr_{\langle \alph \rangle}^{\langle \alph, \beta \rangle} w(\mathbf{j})$
can come from $\sum_{\ell=0}^{p-1} \alph^\ell (e_\mathbf{j})$
when $\{j_1,\ldots, j_p\} = \beta^k(\Delta)$ for some $k$.
But then
$\alph (e_\mathbf{j}) = e_\mathbf{j}$, and so $w(\mathbf{j}) \in M_1$.

Suppose that $c=p-1$. Then $\mathbf{j} = (2,3,\ldots,p,j_p)$ where
$j_p \in \{(m-1)p+1,\ldots, (m-1)p+p\}$ for some $m$ such that $2 \le m \le p$ and
\[ w(\mathbf{j}) = \delta(e_1 \wedge e_2 \wedge \cdots \wedge e_p \wedge (e_{(m-1)p+1}
+\cdots+ e_{(m-1)p+p})). \]
It now follows from~\eqref{eq:deltaProd} at the end of Section~\ref{sec:hook} that
$w(\mathbf{j}) = \delta(e_1 \wedge e_2 \wedge \cdots \wedge e_p) \wedge
(e_{(m-1)p+1} + \cdots + e_{(m-1)p+p}) \in M_{p-1}$. This completes the proof of~(iii).
\end{proof}

In the proof of Proposition~\ref{prop:subspace_induction} below we use
the bilinear form on $M$ defined on the monomial basis of $M$ by
\[ (e_\mathbf{i}, e_{\mathbf{i}'}) = \begin{cases}
  1 & \text{if $\mathbf{i} = \mathbf{i}'$,} \\
  0 & \text{if $\mathbf{i} \not= \mathbf{i}'$,}
  \end{cases} \]
where $\mathbf{i}, \mathbf{i}' \in I$.

\begin{proof}[Proof of Proposition~\ref{prop:subspace_induction}]
We work by induction on $c$. By induction we may assume that
$u \in V_c^{\langle \alph \rangle}$. For $\mathbf{i} \in I$ let
$\mu_\mathbf{i} = (u, e_\mathbf{i})$, so we have
\[ u = \sum_{\mathbf{i} \in I_c \cup \cdots \cup I_p} \mu_\mathbf{i} e_\mathbf{i}. \]
It suffices to find $v \in V_c^{\langle \alph \rangle}$
such that the coefficient of $e_1 \wedge e_{2} \wedge \cdots \wedge e_{p}$
in $\Tr_{\langle \alph \rangle}^{\langle \alph,\beta\rangle} v$
is zero and $u - v \in V_{c+1}$.

\subsubsection*{Case $c=0$}
Let
\[ v' = \sum_{m=2}^p \mu_{((m-1)p+1,\ldots,(m-1)p+p)} z_m \]
where $z_m$ is as defined in Lemma~\ref{lemma:zm}.
Note that the sets $\{(m-1)p+1,\ldots,(m-1)p+p\}$ where $2 \le m \le p$ are the singleton
orbits of $\langle \alph \rangle$ on the set of all $p$-subsets of
$\{1,\ldots, p^2\} \backslash \Delta$.
Let $B^{(1)}, \ldots, B^{(S)}$ be representatives for the
remaining orbits of size $p$, and let $\mathbf{j}^{(1)}, \ldots, \mathbf{j}^{(S)} \in J$ be
the corresponding multi-indices.
Let
\[ v'' =\sum_{s=1}^S \mu_{\mathbf{j}^{(s)}} w(\mathbf{j}^{(s)}) \]
where $w(\mathbf{j})$ is as defined in Lemma~\ref{lemma:wj}
and let $v = v' + v''$. It follows from
Lemmas~\ref{lemma:zm} and~\ref{lemma:wj} that $v \in V^{\langle \alpha\rangle}$
and that
the coefficient of $e_1 \wedge e_{2} \wedge \cdots \wedge e_{p}$
in $\Tr_{\langle \alph \rangle}^{\langle \alph,\beta\rangle} v$ is zero.

Let $\mathbf{j} \in J \cap I_0$. Suppose that $\{j_1,\ldots,j_p\}$ is not fixed by $\alpha$.
There exists a unique set $B^{(t)}$ and a unique $q$
such that $0\le q< p$ and $\alpha^q(\{j_1,\ldots,j_p\}) = B^{(t)}$.
Note that $\alpha^q e_\mathbf{j} = \pm e_{\mathbf{j}^{(t)}}$ where
the sign is determined by the permutation that
puts $\bigl(\alpha^q(j_1), \ldots, \alpha^q(j_p)\bigr)$
in increasing order. By Lemma~\ref{lemma:E_basis} there 
is a basis of $M^{\langle \alpha \rangle}$ in which $\Tr_1^{\langle \alpha \rangle} e_{\mathbf{j}}$
is the unique basis element involving $e_\mathbf{j}$. Since \hbox{$\alpha u = u$}, it
follows that
$(u, e_\mathbf{j}) = (u, \alpha^qe_\mathbf{j}) = (u, \pm e_{\mathbf{j}^{(t)}})$. We also have
$(v, e_{\mathbf{j}^{(t)}}) = (v, \alpha^{-q} e_{\mathbf{j}^{(t)}})
= \pm (v, e_\mathbf{j})$ where the sign is as before. By Lemmas~\ref{lemma:zm}(ii) and \ref{lemma:wj}(ii),
we have $(u, e_{\mathbf{j}^{(t)}}) = \mu_{\mathbf{j}^{(t)}} = (v, e_{\mathbf{j}^{(t)}})$. It follows that $(u,e_\mathbf{j}) = (v,e_\mathbf{j})$. Therefore $u-v \in V_1$.

\subsubsection*{Case $1 \le c \le p-1$.}
Let $X$ be the set of all $(p-c)$-subsets of $\{1, \ldots, p^2\} \backslash \Delta$.
Each orbit of $\langle \alph \rangle$ on $X$ has size $p$.
Let $B^{(1)}, \ldots, B^{(S)}$ be representatives for these orbits and
let $\mathbf{k}^{(1)}, \ldots, \mathbf{k}^{(S)} \in I^{(p-c)}$
be the corresponding multi-indices. Define
\[ D = \bigl\{(i_1,\ldots,i_c) \in I^{(c)}: i_1,\ldots,i_c \in
\Delta \backslash \{1\}\bigr\}.\] 
For each $\mathbf{k}^{(s)}$ define
\[ v_{\mathbf{k}^{(s)}} = \sum_{\mathbf{d} \in D} \mu_{\mathbf{d} : \mathbf{k}^{(s)}}
w(\mathbf{d} : \mathbf{k}^{(s)}) \]
where $\mathbf{d} : \mathbf{k}^{(s)} \in I$ is the multi-index obtained
by concatenating $\mathbf{d}$ and $\mathbf{k}^{(s)}$.
Let
\[ v = \sum_{s=1}^S v_{\mathbf{k}^{(s)}}. \]
It follows from parts (i) and (iii) of Lemma~\ref{lemma:wj}
that $v \in V_c^{\langle \alph \rangle}$ and that the coefficient of
$e_1 \wedge e_{2} \wedge \cdots \wedge e_{p}$ in $\Tr_{\langle \alph \rangle}^{\langle
\alph, \beta \rangle} v$ is zero. It remains to show that $u-v \in V_{c+1}$.

Let $\mathbf{i} \in I_c$. There exists a unique set $B^{(t)}$ and a unique $q$
such that $0\le q \le p-1$ and $\alpha^q( \{i_{c+1}, \ldots, i_p\} ) = B^{(t)}$.
Let $\mathbf{g} = (g_1,\ldots,g_c)$
be the multi-index corresponding to the
set $\alpha^q (\{i_1,\ldots,i_c\})\subseteq \Delta$.
We have $\alpha^q e_\mathbf{i} = \pm e_{\mathbf{g} : \mathbf{k}^{(t)}}$, where the
sign is determined by the permutation putting $\bigl( \alpha^q(i_1), \ldots,
\alpha^q(i_p) \bigr)$ into increasing order.
As in the previous case we have $(u, e_\mathbf{i}) = (u, \alpha^q e_\mathbf{i})
= \pm (u, e_{\mathbf{g} : \mathbf{k}^{(t)}})$ and
$(v, e_{\mathbf{g} : \mathbf{k}^{(t)}}) = (v, \alpha^{-q} e_{\mathbf{g} : \mathbf{k}^{(t)}})
= \pm (v, e_{\mathbf{i}})$ where the signs agree.
Therefore  $(u, e_\mathbf{i}) = (v, e_\mathbf{i})$,
if and only if $(u, e_{\mathbf{g} : \mathbf{k}^{(t)}}) = (v, e_{\mathbf{g} : \mathbf{k}^{(t)}})$.
By Lemma~\ref{lemma:wj}(ii) we have
\begin{align} (v, e_{\mathbf{g} : \mathbf{k}^{(t)}})
&= \sum_{\mathbf{d} \in D}
\mu_{\mathbf{d} : \mathbf{k}^{(t)}} \nonumber
\left ( \, 
 \sum_{\ell=0}^{p-1} \alpha^{\ell} \bigl( \delta(e_1 \wedge e_\mathbf{d}) \wedge e_{\mathbf{k}^{(t)}}
 \bigr), e_{\mathbf{g} : \mathbf{k}^{(t)}} \right ) \\
&= \sum_{\mathbf{d} \in D}
\mu_{\mathbf{d} : \mathbf{k}^{(t)}} \nonumber
(\delta(e_1 \wedge e_\mathbf{d}) \wedge e_{\mathbf{k}^{(t)}}, e_{\mathbf{g} : \mathbf{k}^{(t)}}) \\
 &= \sum_{\mathbf{d} \in D} \mu_{\mathbf{d} : \mathbf{k}^{(t)}}
(\delta(e_1 \wedge e_\mathbf{d}) , e_\mathbf{g} ). \label{eq:veg}
\end{align}
On the other hand, by~\eqref{eq:std} in Section~\ref{sec:hook}
and Lemma~\ref{lemma:rewrite} we have
\[ u = \sum_{\mathbf{d} \in D} \mu_{\mathbf{d} : \mathbf{k}^{(t)}}
\delta(e_1 \wedge e_\mathbf{d} \wedge e_{\mathbf{k}^{(t)}}) + u' \]
where $u'$ is an $F$-linear combination of elements
$\delta(e_1 \wedge e_\mathbf{j})$ of the standard basis for
$\mathbf{j} \in J'$, where
\[ J' = \bigl\{ \mathbf{j} \in (I_{c} \cup \cdots \cup I_p) \cap J :
\text{$\mathbf{j} \not= \mathbf{d} : \mathbf{k}^{(t)}$ for any
$\mathbf{d} \in D$} \bigr\}. \]
Note that
$\bigl(e_{\mathbf{g} : \mathbf{k}^{(t)}}, \delta(e_1 \wedge e_\mathbf{j})\bigr)$
is zero for all $\mathbf{j} \in J'$.
(This is clear if $g_1 \not=1$, since
then
$\bigl(e_{\mathbf{g} : \mathbf{k}^{(t)}}, \delta(e_1 \wedge e_\mathbf{j})\bigr)$
= $(e_{\mathbf{g} : \mathbf{k}^{(t)}}, e_\mathbf{j}) = 0$.
If $g_1 = 1$ then $\mathbf{k}^{(t)}$ must be a subsequence of 
$\mathbf{j}$; since $\mathbf{j} \in I_c \cup \cdots \cup I_p$, it follows
that $\mathbf{j} \in I_c$, all the entries of $\mathbf{j}$ not in $\mathbf{k}^{(t)}$ lie in $\Delta$
and so  $\mathbf{j} = \mathbf{d} : \mathbf{k}^{(t)}$
where $\mathbf{d} \in D$. 
 But then $\mathbf{j} \not\in J'$.)
Hence $(e_{\mathbf{g} : \mbf{k}^{(t)}} , u') = 0$.
It follows from~\eqref{eq:deltaProd} at the end of Section~\ref{sec:hook} that
\[ u \in \sum_{\mathbf{d} \in D} \mu_{\mathbf{d} : \mathbf{k}^{(t)}}
\delta(e_1 \wedge e_\mathbf{d}) \wedge e_{\mathbf{k}^{(t)}} + u' + M_{c+1}\]
and so
\[ (u, e_{\mathbf{g} : \mathbf{k}^{(t)}}) =
\sum_{\mathbf{d} \in D}
 \mu_{\mathbf{d} : \mathbf{k}^{(t)}} (\delta(e_1 \wedge e_\mathbf{d}), e_\mathbf{g}).
\]
Comparing this with~\eqref{eq:veg} we get
$(v, e_\mathbf{g : k^{(t)}}) = (u, e_\mathbf{g : k^{(t)}}) $, as required. Hence $u-v\in V_{c+1}$.
\end{proof}

The proof of Proposition~\ref{prop:small_elem} is now complete.

\section{Open problems}
\label{sec:problems}

We end with some open problems suggested by our three main theorems.

\subsection*{Sylow subgroups}
Say that a group $G$ is $p$-\emph{elementarily large} if it has the property in Theorem~\ref{thm:char}
that if $Q$ is a $p$-subgroup of $G$ containing a $G$-conjugate of every elementary abelian
$p$-subgroup of $G$ then $Q$ is a Sylow $p$-subgroup of $G$. As remarked in the introduction,
not every group is $p$-elementarily large. For example, if $G$ has a
quaternionic Sylow $2$-subgroup then
any non-trivial $2$-subgroup of $G$ contains
the unique elementary abelian $2$-subgroup of $G$ up to $G$-conjugacy. An abelian
group is $p$-elementarily large if and only if its $p$-Sylow subgroup is elementary abelian.

\begin{problem}
Let $p$ be a prime.
Find sufficient conditions for a finite group $G$ to be $p$-elementarily large. In
particular, which finite simple groups are $p$-elementarily large?
\end{problem}

\subsection*{Vertices} 

By Theorem~\ref{thm:vertex}, the following
conjecture holds when $r=p$, $k\equiv1$ mod $p$ and $k \not\equiv 1$ mod $p^2$.

\begin{conjecture}\label{conj:vertex} Let $p$ be an odd prime.
The hook Specht module $S^{(kp-r,1^r)}$
has a Sylow $p$-subgroup of $\S_{kp}$ as a vertex.
\end{conjecture}

We hope that this conjecture will motivate new methods for computing vertices.
The following two examples show some of the limitations of the main methods
in this paper. Take $p=3$. Let $E=E(0,0,1)$; thus $E$ is an elementary abelian $3$-subgroup of
$\S_{27}$ acting regularly on the set $\{1,\ldots,27\}$.
Since the dimension of $\bigwedge^3 S^{(26,1)}$ is $\binom{26}{3}$, which is coprime to $3$,
$\bigwedge^3 S^{(26,1)}$ 
has a Sylow $3$-subgroup of $\S_{27}$ as a vertex. However calculations using {\sc Magma} show that
$(\bigwedge^3 S^{(26,1)})(E) = 0$, so the Brauer correspondence
fails to detect that $E$ is contained in a vertex of $S^{(24,1^3)}$.
Similarly if $E = E(1,1) \le \S_{12}$ then
the $FE(1,1)$-module $S^{(9,1^3)}\res_{E(1,1)}$
is generically free, but by Theorem~\ref{thm:vertex}, $S^{(9,1^3)}$ has a
Sylow $3$-subgroup of $\S_{12}$ as a vertex.

\subsection*{Sources}
If $V$ is an $FG$-module with vertex $Q$ and $V$ is a direct
summand of $U\ind^G$ where $U$ is an indecomposable $FQ$-module, then $U$ is said to
be a \emph{source} of $V$;
the module $U$ is well-defined up to conjugacy by
$\mathrm{N}_G(Q)$. If $n$
is not divisible by $p$ then $S^{(n-1,1)}$ is a direct summand of
the permutation module $\langle e_1,\ldots, e_n\rangle$ and it easily
follows that $\bigwedge^r S^{(n-1,1)}$ has trivial source
 whenever it is indecomposable.
When $p=2$ and $n$ is even it follows from the result of Murphy and Peel mentioned
in the introduction that the source of
$S^{(n-r,1^r)}$ is $S^{(n-r,1^r)}\res_{P}$ where $P$ is a Sylow $2$-subgroup of $\S_n$.
The indirect methods used in this paper give no information about the source of
$S^{(kp-p,1^p)}$
when $p$ is odd.

\begin{problem}
Determine the source of $S^{(kp-p,1^p)}$ when $p$ is odd.
\end{problem}

\begin{problem}
Is there an indecomposable
Specht module with a vertex properly contained in the defect group
of its $p$-block that does not have trivial source?
\end{problem}

\def\cprime{$'$} \def\Dbar{\leavevmode\lower.6ex\hbox to 0pt{\hskip-.23ex
  \accent"16\hss}D} \def\cprime{$'$}
\providecommand{\bysame}{\leavevmode\hbox to3em{\hrulefill}\thinspace}
\providecommand{\MR}{\relax\ifhmode\unskip\space\fi MR }
\providecommand{\MRhref}[2]{%
  \href{http://www.ams.org/mathscinet-getitem?mr=#1}{#2}
}
\providecommand{\href}[2]{#2}
\renewcommand{\MR}[1]{\relax}

\end{document}